\documentclass[a4paper,twoside]{article}

\usepackage[dvipsnames]{xcolor}

\newcommand{\tcr}[1]{\textcolor{black}{#1}}

\usepackage{xspace}
\usepackage[ulem=normalem,draft]{changes}

\usepackage{amsthm,amsmath,amssymb}
\usepackage[bookmarks=true,bookmarksopen=true]{hyperref}
\usepackage{graphicx}
\usepackage{amssymb,amsmath}
\usepackage{url}
\usepackage{xspace}

\usepackage{xspace}

\numberwithin{equation}{section}
\theoremstyle{plain}
\newtheorem{theorem}{Theorem}[section]
\newtheorem{lemma}[theorem]{Lemma}
\newtheorem{corollary}[theorem]{Corollary}

\newtheorem{remark}[theorem]{Remark}

\newcommand{\dx}{\,\mathrm{d}x}
\newcommand{\dt}{\,\mathrm{d}t}
\newcommand{\rd}{{\hat r}}
\newcommand{\pd}{{\hat p}}
\newcommand{\gammab}{{\bar\gamma}}
\newcommand{\Pb}{\mbox{\rm (P)}\xspace}
\newcommand{\PbM}{\mbox{\rm (P$_M$)}\xspace}
\newcommand{\Pbr}{\mbox{\rm (P$_r$)}\xspace}
\newcommand{\Pbg}{\mbox{\rm (P$_\gamma$)}\xspace}
\newcommand{\Pbgp}{\mbox{\rm (P$_{\gamma'}$)}\xspace}
\newcommand{\Pbgk}{\mbox{\rm (P$_{\gamma_k}$)}\xspace}

\newcommand{\Qbg}{\mbox{\rm (Q$_\gamma$)}\xspace}
\newcommand{\uad}{{U_{ad}}}
\newcommand{\kad}{{K_{ad}}}
\newcommand{\uadg}{{U_\gamma}}
\newcommand{\uadgk}{{U_{\gamma_k}}}
\newcommand{\uadgp}{{U_{\gamma'}}}
\newcommand{\proj}{\operatorname{Proj}}
\newcommand{\sign}{\operatorname{sign}}
\newcommand{\supp}{\operatorname{supp}}

\numberwithin{equation}{section}
\numberwithin{theorem}{section}

\begin{document}

\title{Optimal Control of Semilinear Parabolic Equations with Non-smooth Pointwise-Integral Control Constraints in Time-Space.\thanks{The first author was supported by \tcr{MCIN/ AEI/10.13039/501100011033/ under research projects MTM2017-83185-P and PID2020-114837GB-I00}. The second was supported by the ERC advanced grant 668998 (OCLOC) under the EU’s H2020 research program.}
}
\author{Eduardo Casas\\
	Departmento de Matem\'{a}tica Aplicada y Ciencias de la Computaci\'{o}n, \\
	 E.T.S.I. Industriales y de Telecomunicaci\'on, \\
	  Universidad de Cantabria, 39005 Santander, Spain \\
	  eduardo.casas@unican.es\\
	 \and Karl Kunisch \\
	Institute for Mathematics and Scientific Computing, \\
	University of Graz, Heinrichstrasse 36, A-8010 Graz, Austria, \\
	and \\
	Johann Radon Institute, Austrian Academy of Sciences, Linz, Austria.\\
    karl.kunisch@uni-graz.at }

\date{}
\maketitle

\begin{abstract}
	This work concentrates  on a class of optimal control problems for semilinear parabolic equations subject to control constraint of the form $\|u(t)\|_{L^1(\Omega)} \le \gamma$ for $t \in (0,T)$. This limits the total control that can be applied to the system at any instant of time. The $L^1$-norm of the constraint leads to sparsity of the control in space, for the time instants when the constraint is active. Due to the non-smoothness of the constraint, the analysis of the control problem requires new techniques. Existence of a solution, first and second order optimality conditions, and regularity of the optimal control are proved. Further, stability of the optimal controls with respect to $\gamma$ is investigated on the basis of different second order conditions.
\end{abstract}



\section{Introduction}
\label{S1}
We study the optimal control problem
\[
\Pb \quad  \inf_{u \in \uad \cap L^\infty(Q)} J(u):= \frac{1}{2}\int_Q (y_u(x,t) - y_d(x,t))^2\dx\dt + \frac{\kappa}{2}\int_Q u(x,t)^2\dx\dt,
\]
where $\kappa > 0$,
\[
\uad = \{u \in L^\infty(0,T;L^1(\Omega)) : \|u(t)\|_{L^1(\Omega)} \le \gamma \text{ for a.a. } t \in (0,T)\}
\]
with $0 < \gamma < +\infty$, and $y_u$ is the solution of the  semilinear parabolic equation
\begin{equation}
	\left\{\begin{array}{ll}\displaystyle\frac{\partial y}{\partial t} + Ay + a(x,t,y) =  u & \mbox{in } Q = \Omega \times (0,T),\\y = 0  \mbox{ on } \Sigma = \Gamma \times (0,T),& y(0) = y_0 \mbox{ in } \Omega.\end{array}\right.
	\label{E1.1}
\end{equation}
with
\[
Ay=-\sum_{i,j=1}^{n}\partial_{x_j}(a_{ij}(x)\partial_{x_i}y).
\]
We assume that $\Omega$ is a bounded, connected, and open subset of $\mathbb{R}^n$, $n = 2$ or 3, with a Lipschitz boundary $\Gamma$, and that  $0 < T < \infty$ is fixed.

The precise conditions on the nonlinearity $a$ will be given below. Suffice it to say at this moment  that strong nonlinearities such as $\exp (y)$, $\sin(y)$, or polynomial nonlinearities with positive leading term of odd degree will be admitted. A first difficulty that arises in treating $\Pb$ relates to the proof of existence of an optimal control. The reader could think of choosing $L^2(Q)$ as the convenient space to prove the existence of a solution because of the coercivity of $J$ on this space and since the constraint defines a closed and convex subset of $L^2(Q)$. However, the selection of controls in $L^2(Q)$ is not appropriate to deal with the non-linearity in the sate equation. Indeed, even if we can prove the existence of a solution of the state equation, its regularity is not enough (it is not an element of $L^\infty(Q)$, in general) to get the differentiability of the relation control to state. Looking at the control constraint and the cost functional, a second possibility is to consider $L^\infty(0,T;L^2(\Omega))$ as control space. But this is not a reflexive Banach space and, consequently, the proof of existence of a solution to \Pb cannot be done by standard techniques. Nevertheless, we can prove existence of solutions in the spaces $L^r(0,T;L^2(\Omega))$ for all $r > \frac{4}{4-n}$. Moreover, all these solutions belong to $L^\infty(Q)$. This leads us to formulate the control problem in $L^\infty(Q)$; see Remark \eqref{R4.2}. To deal with the non-linearity of the state equation in the proof of a solution to \Pb in $L^\infty(Q)$, one approach consists in introducing artificial bound constraints on the control and prove that they are inactive as the artificial constraint parameter is large enough; see, for instance, \cite{CMR2017}. In our case, this would lead to two control constraints with two Lagrange multipliers in the dual of $L^\infty$. This makes the proof of boundedness of the optimal control very difficult. In this work we avoid such a technique and rather modify (truncate) the non-linear term of the state equation and prove that for a large truncation parameter the cut off is not active on the optimal state.

A second difficulty results from the non-differentiability of the constraint on the control in the definition  of $U_{ad}$. This is a natural constraint since it models a volumetric restriction, which represents a limit to the total amount of control acting at any time $t$. This technological constraint is an alternative to pointwise or to energy constraints which have been considered previously in the literature. Moreover, the $L^1-$ norm in space leads to a spatially sparsifying effect for the solutions. It is different from the type of sparsification which results when considering such terms in the cost. While for the former, sparsification takes place only after the control becomes active, for the latter it takes place regardless of the norm of the control.  For problem $\Pb$ the  sparsity effect is described  by the level set characterized by the functional values of the adjoint state at the height of the supremum norm of the multiplier associated to the control constraint in \Pb; see Corollary \ref{C3.3}. We point out that while  the $L^2$ norm appearing in the cost influences the optimal solution, it  does not eliminate the sparsifying effect of $L^1-$terms, regardless of whether they appear  in the cost or as a constraint. The literature on problems with an $L^1$ or measure-valued norm in the cost is quite rich, so we can only give selected references which consider evolutionary problems
\cite{Casas-Chrysafinos2016,CCK2013,CHW2017,Casas-Kunisch2016,Casas-Kunisch2017,CMR2017,CRT2014,CVZ2015,GKS2021,HSW2012,KPV2014,KTV2016,LVW2020}. In all these papers, either there are no control constraints or they are box constraints. In \cite{GuMa00b}, the authors study a control problem for the evolutionary Navier-Stokes system under the smooth control constraint $\|u(t)\|^2_{L^2(\Omega)} \le 1$, which is smooth and not sparsifying. In \cite{Casas-Kunisch2021}, the control of the 2d evolutionary Navier-Stokes system is analyzed, where the controls are measured valued functions subject to the constraint $\|u(t)\|_{M(\Omega)} \le \gamma$.

The structure of the paper is the following. The analysis of the state equation and its first and second derivatives with respect to the controls is carried out in Section 2. Here special attention is paid to the $L^\infty(Q)$ regularity of the state variable. In Section 3 first order optimality conditions are derived and the structural properties of the involved functions are analyzed. In particular, the regularity of the optimal control is proved, which is a crucial point for the numerical analysis of the control problem. The proof of existence of an optimal control is given in Section 4. Section 5 is devoted to necessary and sufficient second order optimality conditions. In the final section, as a consequence of the second order condition, H\"older and Lipschitz stability of local solutions with respect to the control bound $\gamma$ is investigated.

\section{Analysis of the state equation}
\label{S2}

In this section we establish the well posedness of the state equation, the regularity of the solution, and the differentiable dependence of the solution with respect to the control. To this end we make the following assumptions.

We assume that $y_0 \in L^\infty(\Omega)$, $a_{ij} \in L^\infty(\Omega)$ for every $1 \le i, j \le n$, and
\begin{equation}
	\Lambda_A|\xi|^2\leq\sum_{i,j=1}^n a_{ij}(x)\xi_i\xi_j\ \ \mbox{
	}\forall \xi\in\mathbb{R}^n \ \mbox{ for a.a. } x\in\Omega
	\label{E2.1}
\end{equation}
for some $\Lambda_A > 0$. We also assume that $a:Q\times\mathbb R\to\mathbb R$ is a Carath\'eodory function of class $C^2$ with respect to the last variable satisfying the following properties:
\begin{align}
	&\exists  C_a \in\mathbb R  : \frac{\partial a}{\partial y}(x,t,y)\geq C_a\ \forall y\in\mathbb R,\label{E2.2}\\
	&a(\cdot,\cdot,0)\in L^{\rd}(0,T;L^{\pd}(\Omega)), \text{ with } \rd, \pd \ge 2 \ \text{ and }\  \frac{1}{\rd} + \frac{n}{2\pd} < 1,\label{E2.3}\\
	&\forall M>0\ \exists C_{a,M}>0 : \left|\frac{\partial^j a}{\partial y^j}(x,t,y)\right|\leq C_{a,M}\ \forall |y|\leq M\mbox{ and } j=1,2,\label{E2.4}\\
	&\begin{array}{l}\!\!\!\forall \rho>0 \text{ and } \forall M>0\ \exists \varepsilon>0 \text{ such that}\\
		\displaystyle\!\!\!\left|\frac{\partial^2 a}{\partial y^2}(x,t,y_1)-\frac{\partial^2 a}{\partial y^2}(x,t,y_2)\right|<\rho\ \ \ \forall|y_1|,|y_2|\leq M \text{ with } \ |y_1-y_2|<\varepsilon,\end{array}\label{E2.5}
\end{align}
for almost all $(x,t) \in Q$.

As usual $W(0,T)$ denotes the Hilbert space
\[
W(0,T) = \{y \in L^2(0,T;H_0^1(\Omega)) : \frac{\partial y}{\partial t} \in L^2(0,T;H^{-1}(\Omega))\}.
\]
We recall that $W(0,T)$ is continuously embedded in $C([0,T];L^2(\Omega))$ and compactly embedded in $L^2(Q)$.

\begin{theorem}
	Under the previous assumptions, for every $u \in L^r(0,T;L^p(\Omega))$ with $\frac{1}{r} + \frac{n}{2p} < 1$ and $r, p \ge 2$ there exists a unique solution $y_u \in L^\infty(Q) \cap W(0,T)$ of \eqref{E1.1}. Moreover, the following estimates hold
	\begin{align}
		&\|y_u\|_{L^\infty(Q)} \le \eta\big(\|u\|_{L^r(0,T;L^p(\Omega))} + \|a(\cdot,\cdot,0)\|_{L^\rd(0,T;L^\pd(\Omega))} + \|y_0\|_{L^\infty(\Omega)}\big),\label{E2.6}\\
		&\|y_u\|_{C([0,T];L^2(\Omega))} + \|y_u\|_{L^2(0,T;H_0^1(\Omega))}\notag\\
		&\hspace{2.4cm}\le K\big(\|u\|_{L^2(Q)} + \|a(\cdot,\cdot,0)\|_{L^2(Q)} + \|y_0\|_{L^2(\Omega)}\big),\label{E2.7}
	\end{align}
	for a monotone non-decreasing function $\eta:[0,\infty) \longrightarrow [0,\infty)$ and some constant $K$ both independent of $u$.
	\label{T2.1}
\end{theorem}

\begin{proof}
	We decompose the state equation into two parts. First, we consider
	\begin{equation}
		\left\{\begin{array}{ll}\displaystyle\frac{\partial z}{\partial t} + Az  =  u & \mbox{in } Q,\\z = 0  \mbox{ on } \Sigma,& z(0) = y_0 \mbox{ in } \Omega.\end{array}\right.
		\label{E2.8}
	\end{equation}
	It is well known that it has a unique solution $z \in W(0,T) \cap L^\infty(Q)$.
	Moreover, we have the estimates
	\begin{align}
		&\|z\|_{W(0,T)} \le C_W(\|u\|_{L^2(Q)} + \|y_0\|_{L^2(\Omega)}),\label{E2.9}\\
		&\|z\|_{L^\infty(Q)} \le C_\infty(\|u\|_{L^r(0,T;L^p(\Omega))} + \|y_0\|_{L^\infty(\Omega)});\label{E2.10}
	\end{align}
	see, for instance, \cite[Chapter III]{Lad-Sol-Ura68}. Now, we define $b:Q \times \mathbb{R} \longrightarrow \mathbb{R}$ as follows
	\[
	b(x,t,s) = \text{\em e}^{-|C_a|t}[a(x,t,\text{\em e}^{|C_a|t}s + z(x,t)) - a(x,t,z(x,t))] + |C_a|s,
	\]
	where $C_a$ is as in \eqref{E2.2}. Then, $b(x,t,0) = 0$ and according to \eqref{E2.2}
	\[
	\frac{\partial b}{\partial s}(x,t,s) = \frac{\partial a}{\partial s}(x,t,\text{\em e}^{|C_a|t}s + z(x,t)) + |C_a| \ge 0.
	\]
	We consider the equation
	\begin{equation}
		\left\{\begin{array}{ll}\displaystyle\frac{\partial w}{\partial t} + Aw + b(x,t,w) =  -\text{\em e}^{-|C_a|t}a(x,t,z(x,t)) & \mbox{in } Q,\\w = 0  \mbox{ on } \Sigma, \ \ w(0) = 0 \mbox{ in } \Omega.&\end{array}\right.
		\label{E2.11}
	\end{equation}
	Due to the properties of $b$, the existence and uniqueness of a solution $w \in L^\infty(Q) \cap W(0,T)$ is well known; see \cite[Theorem 5.5]{Troltzsch2010}. Moreover, the following estimates hold
	\begin{align}
		&\|w\|_{W(0,T)} \le C_b(\|a(\cdot,\cdot,z)\|_{L^2(Q)} + \|b(\cdot,\cdot,w)\|_{L^2(Q)}),\label{E2.12}\\
		&\|w\|_{L^\infty(Q)} \le C'_\infty\|a(\cdot,\cdot,z)\|_{L^\rd(0,T;L^\pd(\Omega))}.\label{E2.13}
	\end{align}
	Denoting $M = \|z\|_{L^\infty(Q)}$ and using \eqref{E2.4} we infer with the mean value theorem
	\begin{align*}
		&|a(x,t,z(x,t))| \le |a(x,t,z(x,t)) - a(x,t,0)| + |a(x,t,0)|\\
		& = \Big|\frac{\partial a}{\partial y}(x,t,\theta(x,t)z(x,t))z(x,t)\Big| + |a(x,t,0)| \le C_{a,M}M + |a(x,t,0)|.
	\end{align*}
	Combining this with \eqref{E2.10} and \eqref{E2.13} we get
	\begin{equation}
		\|w\|_{L^\infty(Q)} \le \sigma(\|u\|_{L^r(0,T;L^p(\Omega))} + \|a(\cdot,\cdot,0)\|_{L^\rd(0,T;L^\pd(\Omega))} + \|y_0\|_{L^\infty(\Omega)}\big),
		\label{E2.14}
	\end{equation}
	for a non-decreasing function $\sigma:[0,\infty) \longrightarrow [0,\infty)$.
	
	If we set $w = \text{\em e}^{-|C_a|t}\psi$ and insert this in \eqref{E2.11}, we infer
	\begin{equation}
		\left\{\begin{array}{ll}\displaystyle\frac{\partial\psi}{\partial t} + A\psi + a(x,t,z(x,t) + \psi) =  0 & \mbox{in } Q,\\\psi = 0  \mbox{ on } \Sigma,\ \ \psi(0) = 0 \mbox{ in } \Omega.&\end{array}\right.
		\label{E2.15}
	\end{equation}
	Adding \eqref{E2.8} and \eqref{E2.15} we deduce that $y_u = z + \psi$ solves \eqref{E1.1}. Moreover, any solution of \eqref{E1.1} is the sum of the solutions of \eqref{E2.8} and \eqref{E2.15}. Since these equations have a unique solution, the uniqueness of $y_u$ follows. Furthermore, \eqref{E2.10} and \eqref{E2.14} imply \eqref{E2.6}.
	
	To prove \eqref{E2.7}, we take $\phi= \text{\em e}^{-|C_a|t}y_u$ and introduce the function $f:Q \times \mathbb{R} \longrightarrow \mathbb{R}$ defined by
	\[
	f(x,t,s) = \text{\em e}^{-|C_a|t}[a(x,t,\text{\em e}^{|C_a|t}s) - a(x,t,0)] + |C_a|s.
	\]
	Then, $\phi$ satisfies
	\begin{equation}
		\left\{\begin{array}{ll}\displaystyle\frac{\partial\phi}{\partial t} + A\phi + f(x,t,\phi) = \text{\em e}^{-|C_a|t}[u -a(x,t,0)] & \mbox{in } Q,\\\phi = 0  \mbox{ on } \Sigma, \ \ \phi(0) = y_0 \mbox{ in } \Omega.&\end{array}\right.
		\label{E2.16}
	\end{equation}
	Since $f(x,t,0) = 0$ and $\frac{\partial f}{\partial s}(x,t,s) \ge 0$, multiplying the above equation by $\phi$, integrating in $\Omega$, and using \eqref{E2.1} we get
	\begin{align*}
		&\frac{1}{2}\frac{d}{dt}\|\phi(t)\|^2_{L^2(\Omega)} + \Lambda_A\int_\Omega|\nabla\phi(t)|^2\dx\\
		&\le \frac{1}{2}\frac{d}{dt}\|\phi(t)\|^2_{L^2(\Omega)} + \sum_{i,j = 1}^n\int_\Omega a_{ij}\partial_{x_i}\phi(t)\partial_{x_j}\phi(t)\dx + \int_\Omega f(x,t,\phi(t))\phi(t)\dx\\
		&= \int_\Omega\text{\em e}^{-|C_a|t}(u - a(x,t,0))\phi\dx \le \big(\|u(t)\|_{L^2(\Omega)} + \|a(\cdot,t,0)\|_{L^2(\Omega)}\big)\|\phi(t)\|_{L^2(\Omega)}.
	\end{align*}
	Estimate \eqref{E2.7} follows from this inequality as usual.
	\qed\end{proof}

We apply Theorem \ref{T2.1} with $p = 2$ and $r \in \big(\frac{4}{4-n},\infty\big]$. Observe that $\frac{1}{r} + \frac{n}{4} < 1$ and $r > 2$. Then, the mapping $G:L^r(0,T;L^2(\Omega)) \longrightarrow L^\infty(Q) \cap W(0,T)$ given by $G(u) = y_u$ solution of \eqref{E1.1} is well defined. We have the following differentiability properties of $G$.

\begin{theorem}
	The mapping $G$ is of class $C^2$. For $u,v,v_1,v_2 \in L^r(0,T;L^2(\Omega))$ the derivatives $z_v = G'(u)v$ and $z_{v_1,v_2} = G''(u)(v_1,v_2)$ are the solutions of the equations
	\begin{align}
		&\left\{\begin{array}{l}\displaystyle\frac{\partial z_v}{\partial t} + A z_v + \frac{\partial a}{\partial y}(x,t,y_u)z_v =  v \ \mbox{ in } Q,\\ z_v = 0  \mbox{ on } \Sigma,\ \ z_v(0) = 0 \mbox{ in } \Omega,\end{array}\right.
		\label{E2.17}\\
		&\left\{\begin{array}{l}\displaystyle\frac{\partial z_{v_1,v_2}}{\partial t} + A z_{v_1,v_2} + \frac{\partial a}{\partial y}(x,t,y_u)z_{v_1,v_2} + \frac{\partial^2a}{\partial y^2}(x,t,y_u)z_{v_1}z_{v_2} = 0 \ \mbox{ in } Q,\\ z_{v_1,v_2} = 0  \mbox{ on } \Sigma,\ \ z_{v_1,v_2}(0) = 0 \mbox{ in } \Omega.\end{array}\right.
		\label{E2.18}
	\end{align}
	\label{T2.2}
\end{theorem}

\begin{proof}
	Let us consider the Banach space
	\[
	\mathcal{Y} =\{y \in L^\infty(Q) \cap W(0,T) : \frac{\partial y}{\partial t} + Ay \in X\},
	\]
	where $X = L^\rd(0,T;L^\pd(\Omega)) + L^r(0,T;L^2(\Omega))$, endowed with the norm
	\[
	\|y\|_{\mathcal{Y}} = \|y\|_{L^\infty(Q)} + \|y\|_{W(0,T)} + \|\frac{\partial y}{\partial t} + Ay\|_X.
	\]
	Now, we define the mapping
	\begin{align*}
		&\mathcal{F}:\mathcal{Y} \times L^\infty(\Omega) \times L^r(0,T;L^2(\Omega)) \longrightarrow X \times L^\infty(\Omega)\\
		&\mathcal{F}(y,w,u) =\big(\frac{\partial y}{\partial t} + A y + a(\cdot,\cdot,y) - u, y(0) - w\big).
	\end{align*}
	We have that $\mathcal{F}$ is of class $C^2$, $\mathcal{F}(y_u,y_0,u) = (0,0)$ for every $u \in L^r(0,T;L^2(\Omega))$, and
	\begin{align*}
		&\frac{\partial\mathcal{F}}{\partial y}(y_u,y_0,u):\mathcal{Y} \longrightarrow X \times L^\infty(\Omega)\\
		&\frac{\partial\mathcal{F}}{\partial y}(y_u,y_0,u)z = \big(\frac{\partial z}{\partial t} + A z + \frac{\partial a}{\partial y}(\cdot,\cdot,y_u)z, z(0)\big)
	\end{align*}
	is an isomorphism. Hence, an easy application of the implicit function theorem proves the result.
	\qed\end{proof}

As a consequence of the above theorem and the chain rule we infer the differentiability of the mapping $J:L^r(0,T;L^2(\Omega)) \longrightarrow \mathbb{R}$. From now on, we assume
\begin{equation}
	y_d \in L^\rd(0,T;L^\pd(\Omega)),
	\label{E2.19}
\end{equation}
where $\rd$ and $\pd$ are defined in \eqref{E2.3}.
\begin{corollary}
	If $r > \frac{4}{4-n}$, then $J$ is of class $C^2$ and its derivatives are given by the expressions
	\begin{align}
		&J'(u)v = \int_Q(\varphi + \kappa u)v\dx\dt,\label{E2.20}\\
		&J''(u)(v_1,v_2) = \int_Q\Big[\big(1 - \frac{\partial^2a}{\partial y^2}(x,t,y_u)\varphi\big)z_{v_1}z_{v_2} + \kappa v_1v_2\Big]\dx\dt,\label{E2.21}
	\end{align}
	where $z_{v_i} = G'(u)v_i$, $i = 1, 2$, and $\varphi \in C(\bar Q) \cap H^1(Q)$ is the solution of the adjoint state equation
	\begin{equation}
		\left\{\begin{array}{l}\displaystyle-\frac{\partial\varphi}{\partial t} + A^*\varphi + \frac{\partial a}{\partial y}(x,t,y_u)\varphi =  y_u - y_d\ \mbox{ in } Q,\\ \varphi = 0  \mbox{ on } \Sigma,\ \ \varphi(T) = 0 \mbox{ in } \Omega.\end{array}\right.
		\label{E2.22}
	\end{equation}
	\label{C2.1}
\end{corollary}
Above $A^*$ denotes the adjoint operator of $A$
\[
A^*\varphi=-\sum_{i,j=1}^{n}\partial_{x_j}(a_{ji}(x)\partial_{x_i}\varphi).
\]

The regularity $\bar\varphi \in C(\bar Q) \cap H^1(Q)$ follows from Theorems III-6.1 and III-10.1 of \cite{Lad-Sol-Ura68}. Moreover, we observe that $J'(u)$ and $J''(u)$ can be extended to continuous linear and bilinear forms $J'(u):L^2(Q) \longrightarrow \mathbb{R}$ and $J''(u): L^2(Q) \times L^2(Q) \longrightarrow \mathbb{R}$ for every $u \in L^r(0,T;L^2(\Omega))$.

\begin{remark}
	Hypotheses \eqref{E2.1}--\eqref{E2.5} are satisfied, for instance, for the nonlinearity $a(y)= \exp(y)$. They are also satisfied for  $a(y)= (y- z_1)(y- z_2)(y- z_3)$ for constants $z_i$, with $i\in \{1,2,3\}$. This latter nonlinearity is known in neurology as Nagumo equation and in physical chemistry as Schl\"ogl model. Formulating the optimal control problem with an $L^1(\Omega)$ constraint implies that one looks for the action of a controlling laser whose optimal support is small; see \cite{GMT2020}.
	\label{R2.1}
\end{remark}

\section{Existence of optimal controls and first order optimality conditions}
\label{S3}

Since the control problem \Pb is not convex, we need to distinguish between local and global minimizers. We call $\bar u$ a local minimizer for \Pb in the $L^r(0,T;L^2(\Omega))$ sense with $r > \frac{4}{4-n}$ if $\bar u \in \uad \cap L^\infty(Q)$ and there exists $\varepsilon > 0$ such that
\begin{equation}
	J(\bar u) \le J(u)\quad \forall u \in B_\varepsilon \cap \uad,
	\label{E3.1}
\end{equation}
where
\[
B_\varepsilon =\{u \in L^r(0,T;L^2(\Omega)) : \|u - \bar u\|_{L^r(0,T;L^2(\Omega))} \le \varepsilon\}.
\]
It is immediate to check that if $\bar u$ is a local minimizer in the $L^r(0,T;L^2(\Omega))$ sense, then it is also a local minimizer in the $L^{r'}(0,T;L^2(\Omega))$ sense for every $r < r' \le \infty$.

\begin{theorem}
	There exists at least one solution of \Pb. Moreover, for every local minimizer $\bar u$ in the $L^r(0,T;L^2(\Omega))$ sense with $r > \frac{4}{4-n}$, there exist $\bar y \in L^2(0,T;H_0^1(\Omega)) \cap L^\infty(Q)$, $\bar\varphi \in C(\bar Q) \cap H^1(Q)$, and $\bar\mu \in L^\infty(Q)$ such that
	\begin{align}
		&\left\{\begin{array}{l}\displaystyle\frac{\partial\bar y}{\partial t} + A\bar y + a(x,t,\bar y) =  \bar u \ \mbox{ in } Q,\\ \bar y = 0  \mbox{ on } \Sigma,\ \ \bar y(0) = y_0 \mbox{ in } \Omega,\end{array}\right.
		\label{E3.2}\\
		&\left\{\begin{array}{l}\displaystyle-\frac{\partial\bar\varphi}{\partial t} + A^*\bar\varphi + \frac{\partial a}{\partial y}(x,t,\bar y)\bar\varphi =  \bar y - y_d\ \mbox{ in } Q,\\ \bar\varphi = 0  \mbox{ on } \Sigma,\ \ \bar\varphi(T) = 0 \mbox{ in } \Omega,\end{array}\right.
		\label{E3.3}\\
		&\int_Q\bar\mu(u - \bar u)\dx\dt \le 0\quad \forall u \in \uad,
		\label{E3.4}\\
		&\bar\varphi + \kappa\bar u + \bar\mu = 0.
		\label{E3.5}
	\end{align}
	\label{T3.1}
\end{theorem}

\begin{proof}
	The proof of existence of a solution for \Pb is postponed to the next section, see Theorem \ref{T4.5}. Given a local minimizer $\bar u$, we take $\bar y$ and $\bar\varphi$ as solutions of \eqref{E3.2} and \eqref{E3.3}, respectively. Using the convexity of $\uad$ and \eqref{E2.20} we get
	\[
	0 \le J'(\bar u)(u - \bar u) = \int_Q(\bar\varphi + \kappa\bar u)(u - \bar u) \dx\dt \quad \forall u \in \uad \cap \tcr{L^\infty(Q)}.
	\]
	Now, given $u \in \uad$ arbitrary, we set $u_k(x,t) = \proj_{[-k,+k]}(u(x,t))$ for $k \ge 1$, thus $\{u_k\}_{k= 1}^\infty \subset \tcr{L^\infty(Q)} \cap \uad$ and $u_k \to u$ in $L^1(Q)$. Then, we can pass to the limit in the inequality $J'(\bar u)(u_k - \bar u) \ge 0$ and, hence, we obtain
	\[
	\int_Q(\bar\varphi + \kappa\bar u)(u - \bar u) \dx\dt \ge 0\ \ \forall u \in \uad.
	\]
	This inequality is equivalent to the fact $-(\bar\varphi + \kappa\bar u) \in \partial I_{\uad}(\bar u) \subset L^\infty(Q)$. Here $\partial I_{\uad}$ denotes the subdifferential of the indicator function $I_{\uad}:L^1(Q) \longrightarrow [0,+\infty]$, which takes the value $I_\uad(u) = 0$ if $u \in \uad$ and $+\infty$ otherwise. Therefore, there exists $\bar\mu \in \partial I_{\uad}$ such that \eqref{E3.4} and \eqref{E3.5} holds.
	\qed\end{proof}

Let us denote by $\proj_{B_\gamma}:L^2(\Omega) \longrightarrow B_\gamma \cap L^2(\Omega)$ the  $L^2(\Omega)$ projection, where $B_\gamma = \{v \in L^1(\Omega) : \|v\|_{L^1(\Omega)} \le \gamma\}$. Then, we have the following consequence of the previous theorem.

\begin{corollary}
	Let $\bar u$, $\bar\varphi$, and $\bar\mu$ satisfy \eqref{E3.2}--\eqref{E3.5}. Then, the following properties hold
	\begin{align}
		&\int_\Omega\bar\mu(t)(v - \bar u(t))\dx \le 0\ \ \forall v \in B_\gamma \text{ and for a.a. } t \in (0,T),\label{E3.6}\\
		&\bar u(t) = \proj_{B_\gamma}\big(-\frac{1}{\kappa}\bar\varphi(t)\big) \text{ for a.a. } t \in (0,T),\label{E3.7}\\
		&\left\{\begin{array}{l}\bar u(x,t)\bar\mu(x,t) = |\bar u(x,t)||\bar\mu(x,t)| \text{ for a.a. } (x,t) \in Q,\\[1 ex] \text{if } \|\bar u(t)\|_{L^1(\Omega)} < \gamma \text{ then } \bar\mu(t) \equiv 0 \text{ in } \Omega \text{ a.e. in } (0,T),\\[1 ex] \text{if } \|\bar u(t)\|_{L^1(\Omega)} = \gamma \text{ and } \bar\mu(t) \not\equiv 0 \text{ in } \Omega,\\[1 ex] \hspace{2cm} \text{ then } \supp(\bar u(t)) \subset \{x \in \Omega :  |\bar\mu(x,t)| = \|\bar\mu(t)\|_{L^\infty(\Omega)}\}.\end{array}\right.\label{E3.8}
	\end{align}
	\label{C3.1}
\end{corollary}

\begin{proof}
	Let us show that \eqref{E3.4} and \eqref{E3.6} are equivalent. Using Fubini's theorem, it is  obvious that \eqref{E3.6} implies \eqref{E3.4}. Let us prove the contrary implication. Let $v \in B_\gamma$ be arbitrary and set
	\[
	I_v = \big\{t \in (0,T) : \int_\Omega\bar\mu(x,t)(v(x) - \bar u(x,t))\dx > 0\Big\}
	\]
	and
	\[
	u(x,t) = \left\{\begin{array}{cl}v(x) &\text{if } t \in I_v,\\\bar u(x,t)&\text{otherwise.}\end{array}\right.
	\]
	Then, $u \in \uad$ and \eqref{E3.4} yields
	\[
	0 \ge \int_Q\bar\mu(x,t)(u - \bar u)\dx\dt = \int_{I_v}\int_\Omega\bar\mu(x,t)(v(x) - \bar u(x,t))\dx.
	\]
	This is only possible if $|I_v| = 0$. In order to prove \eqref{E3.7} we use \eqref{E3.5} and \eqref{E3.6} to get
	\[
	\int_\Omega\big(-\frac{1}{\kappa}\bar\varphi(t) - \bar u(t))(v - \bar u(t)) \le 0 \ \ \forall v \in B_\gamma \cap L^2(\Omega) \text{ and for a.a. } t \in (0,T).
	\]
	Since $B_\gamma \cap L^2(\Omega)$ is a convex and closed subset of $L^2(\Omega)$, the above inequality is the well known characterization of \eqref{E3.7}.
	
	Let us prove the first statement of \eqref{E3.8}. Take $u(x,t) = \sign(\bar\mu(x,t))|\bar u(x,t)|$. Then, $u \in \uad$ and with \eqref{E3.4} we obtain
	\[
	\int_Q|\bar\mu(x,t)||\bar u(x,t)|\dx\dt = \int_Q\bar\mu(x,t)u(x,t)\dx\dt \le \int_Q\bar\mu(x,t)\bar u(x,t)\dx\dt,
	\]
	which proves the desired identity. We prove the second statement of \eqref{E3.8}. For every $\varepsilon > 0$ we define
	\[
	I_\varepsilon = \{t \in (0,T) : \|\bar u(t)\|_{L^1(\Omega)} \le \gamma - \varepsilon\}.
	\]
	Denote $B_\varepsilon$ the closed ball of $L^1(\Omega)$ centered at 0 and radius $\varepsilon$. Take $v \in B_\varepsilon$ arbitrary. Then, we have that $v+\bar u(t) \in B_\gamma$ for $t \in I_\varepsilon$, and \eqref{E3.6} yields
	\[
	\int_\Omega\bar\mu(x,t)v(x)\dx \le 0\ \ \forall v \in B_\varepsilon \text{ and } t \in I_\varepsilon,
	\]
	which implies that $\bar\mu(t) \equiv 0$ in $\Omega$ for $t \in I_\varepsilon$. Since $\varepsilon > 0$ is arbitrary, we infer the second statement of \eqref{E3.8}. Let us prove the third statement. Under the assumption $\|\bar u(t)\|_{L^1(\Omega)} = \gamma$ and $\bar\mu(t) \not\equiv 0$ in $\Omega$. For every $\varepsilon > 0$ and $t \in (0,T)$ we consider the sets
	\begin{align*}
		&\Omega^\varepsilon(t) = \{x \in \Omega : |\bar u(x,t)| > \varepsilon \text{ and } |\bar\mu(x,t)| < \|\bar\mu(t)\|_{L^\infty(\Omega)} - \varepsilon\},\\
		&\tilde\Omega^\varepsilon(t) = \{x \in \Omega : |\bar\mu(x,t)| > \|\bar\mu(t)\|_{L^\infty(\Omega)} - \varepsilon\}.
	\end{align*}
	We are going to prove that $|\Omega^\varepsilon(t)| = 0$ for almost all $t \in (0,T)$. Assume that $|\Omega^\varepsilon(t)| > 0$ for some $\varepsilon > 0$ and $t \in (0,T)$. Since $|\tilde\Omega^\varepsilon(t)| > 0$ by definition of the essential supremum, we can find two sets $E \subset \Omega^\varepsilon(t)$ and $F \subset \tilde\Omega^\varepsilon(t)$ such that $|E| = |F| > 0$. We define the control
	\[
	v(x) = \left\{\begin{array}{cl}\bar u(x,t) - \varepsilon\sign(\bar u(x,t))&\text{if } x \in E,\\\bar u(x,t) + \varepsilon\sign(\bar u(x,t))&\text{if } x \in F,\\\bar u(x,t)&\text{otherwise.}\end{array}\right.
	\]
	Since $\|\bar u(t)\|_{L^1(\Omega)} = \gamma$, we get
	\[
	\|v\|_{L^1(\Omega)} = \int_E|\bar u(x)|\dx - \varepsilon|E| + \int_F|\bar u(x)|\dx + \varepsilon|F| + \int_{\Omega\setminus(E\cup F)}|\bar u(x)|\dx = \gamma.
	\]
	Moreover, we get with the first statement of \eqref{E3.8}
	\[
	\int_\Omega\bar\mu(x,t)(v(x)- \bar u(x,t))\dx = -\varepsilon\int_E|\bar\mu(x,t)|\dx + \varepsilon\int_F|\bar\mu(x,t)|\dx > 0,
	\]
	which contradicts \eqref{E3.6} unless it is satisfied for a set of points $t$ of zero Lebesgue measure. Taking
	\[
	\Omega(t) = \{x \in \Omega : |\bar u(x,t)| > 0 \text{ and } |\bar\mu(x,t)| < \|\bar\mu(t)\|_{L^\infty(\Omega)}\},
	\]
	since $\varepsilon > 0$ was arbitrary, we deduce that $|\Omega(t)| = 0$ for almost all $t \in (0,T)$. This implies that $\supp(\bar u(t)) \subset \{x \in \Omega : |\bar\mu(x,t)| = \|\bar\mu(t)\|_{L^\infty(\Omega)}\}$.
	\qed\end{proof}

\begin{remark}
	Let us observe that the first statement of \eqref{E3.8} and \eqref{E3.5} imply
	\[
	|\bar\varphi(x,t)| = \kappa|\bar u(x,t)| + |\bar\mu(x,t)|.
	\]
	This yields
	\[
	\|\bar\varphi(t)\|_{L^1(\Omega)} = \kappa\|\bar u(t)\|_{L^1(\Omega)} + \|\bar\mu(t)\|_{L^1(\Omega)}.
	\]
	From this identity and the second statement of \eqref{E3.8} we infer that $\bar\mu(t) \not\equiv 0$ in $\Omega$ if and only if $\|\bar\varphi(t)\|_{L^1(\Omega)} > \kappa\gamma$.
	\label{R3.1}
\end{remark}

\begin{remark}
	From \eqref{E3.8} we deduce that $\bar\mu(x,t) \in \|\bar\mu(t)\|_{L^\infty(\Omega)}\,\partial|\cdot|(\bar u(x,t))$ for almost every point $(x,t) \in Q$.
	\label{R3.2}
\end{remark}

\begin{corollary}
	Let $\bar u \in \uad \cap L^\infty(\Omega)$ satisfy \eqref{E3.5} and \eqref{E3.8}. Then, the following identities are satisfied
	\begin{align}
		&\bar u(x,t) = -\frac{1}{\kappa}\sign(\bar\varphi(x,t))\big(|\bar\varphi(x,t)| - \|\bar\mu(t)\|_{L^\infty(\Omega)}\big)^+\notag\\
		&\hspace{0.95cm}=-\frac{1}{\kappa}\left\{\Big[\bar\varphi(x,t) + \|\bar\mu(t)\|_{L^\infty(\Omega)}\Big]^- + \Big[\bar\varphi(x,t) - \|\bar\mu(t)\|_{L^\infty(\Omega)}\Big]^+\right\}.
		\label{E3.9}
	\end{align}
	Moreover, the regularity $\bar u \in H^1(Q)$ and $\bar\mu \in H^1(Q)$ hold.
	\label{C3.2}
\end{corollary}

\begin{proof}
	If $\|\bar\mu(t)\|_{L^\infty(\Omega)} = 0$, then $\bar u(x,t) = -\frac{1}{\kappa}\bar\varphi(x,t)$ follows from \eqref{E3.5}, which coincides with the identity \eqref{E3.9}. Now, we assume that $\|\bar\mu(t)\|_{L^\infty(\Omega)} > 0$. Using \eqref{E3.8} we obtain that $\|\bar u(t)\|_{L^1(\Omega)} = \gamma$. Then, the third statement of \eqref{E3.8} implies that $|\bar\mu(x,t)| = \|\bar\mu(t)\|_{L^\infty(\Omega)}$ if $|\bar u(x,t)| > 0$. We distinguish three cases.
	
	i) If $\bar u(x,t) > 0$, \eqref{E3.5} and the first statement of \eqref{E3.8} leads to $\bar u(x,t) = -\frac{1}{\kappa}(\bar\varphi(x,t) + \|\bar\mu(t)\|_{L^\infty(\Omega)})$, which coincides with the expression \eqref{E3.9}.
	
	ii) If $\bar u(x,t) = 0$, using again \eqref{E3.5} we get $|\bar\varphi(x,t)| = |\bar\mu(x,t)| \le \|\bar\mu(t)\|_{L^\infty(\Omega)}$. Then, the identity \eqref{E3.9} holds.
	
	iii) If $\bar u(x,t) < 0$, from the first statement of \eqref{E3.8} and \eqref{E3.5} we infer that $\bar u(x,t) = -\frac{1}{\kappa}(\bar\varphi(x,t) - \|\bar\mu(t)\|_{L^\infty(\Omega)})$. Then, \eqref{E3.9} holds too.
	
	The spatial regularity $\bar u \in L^2(0,T;H_0^1(\Omega))$ is an immediate consequence of \eqref{E3.9} and the fact that $\bar\varphi \in H^1(Q)$. For the temporal regularity of $\bar u$, we first observe
	\begin{align*}
		&\|\bar u(t) - \bar u(t')\|_{L^2(\Omega)}\\
		&= \|\proj_{B_\gamma}(-\frac{1}{\kappa}\bar\varphi(t)) - \proj_{B_\gamma}(-\frac{1}{\kappa}\bar\varphi(t'))\|_{L^2(\Omega)} \le \frac{1}{\kappa}\|\bar\varphi(t) - \bar\varphi(t')\|_{L^2(\Omega)}.
	\end{align*}
	Since $\bar\varphi:[0,T] \longrightarrow L^2(\Omega)$ is absolutely continuous, using the above inequality we infer that $\bar u:[0,T] \longrightarrow L^2(\Omega)$ is also absolutely continuous. Moreover, the same inequality yields $\|\bar u'(t)\|_{L^2(\Omega)} \le \frac{1}{\kappa}\|\bar\varphi'(t)\|_{L^2(\Omega)}$ and $\bar u \in W^{1,2}(0,T;L^2(\Omega))$. All together, this implies that $\bar u \in H^1(Q)$. The regularity of $\bar\mu$ follows from \eqref{E3.5}.
	\qed\end{proof}

\begin{corollary}
	Let $\bar u$ be as in Corollary \ref{C3.2}. Then, we have the following property
	\begin{equation}
		\bar u(x,t) = 0 \text{ if and only if } |\bar\varphi(x,t)| \le \|\bar\mu(t)\|_{L^\infty(\Omega)}.
		\label{E3.10}
	\end{equation}
	\label{C3.3}
\end{corollary}
This corollary is a straightforward consequence of \eqref{E3.9}.
\begin{theorem}
	There exists a constant $K_\infty > 0$ independent of $\gamma$ such that $\|\bar u\|_{L^\infty(Q)} \le K_\infty$ for every global minimizer $\bar u$ of \Pb. In addition, if we set $\gamma_0 = K_\infty|\Omega|$, then for every $\gamma > \gamma_0$ and every solution $\bar u$ of \Pb we have $\|\bar u(t)\|_{L^1(\Omega)} < \gamma$ for almost every $t$ and $\bar u = -\frac{1}{\kappa}\bar\varphi$.
	\label{T3.2}
\end{theorem}

To prove this theorem, we can argue as in the proof of Theorem \ref{T4.4} below to deduce the existence of $K_\infty > 0$ independent of $\gamma$ such that $\|\bar u\|_{L^\infty(Q)} \le K_\infty$. The last statement is a straightforward consequence of this estimate and the definition of $\gamma_0$.

\section{Proof of existence of a solution for \Pb}
\label{S4}

The proof of existence of a solution of \Pb can not be performed by the classical method of calculus of variations due to the lack of boundedness of $\uad$ in $L^\infty(\Omega)$ and the non coercivity of $J$ on this space. One can try to prove the existence of a solution $\bar u$ of \Pb in $L^2(Q)$ and then to deduce that $\bar u \in L^\infty(Q)$ from the optimality conditions. However, the differentiability of $J$ in $L^2(Q)$ can fail due to the nonlinearity of the state equation. To overcome this difficulty we are going to truncate the nonlinear term $a(x,t,y)$ as follows. For every $M > 0$ we define the function $f_M:\mathbb{R} \longrightarrow \mathbb{R}$ by
\[
f_M(s) = \left\{\begin{array}{cl} M + 1& \text{if } s > M + 1,\\s+ (M - s)^2 + (M - s)^3& \text{if } M \le s \le M + 1,\\s&\text{if } -M < s < +M,\\s - (M + s)^2 - (M+s)^3&\text{if } -M - 1 \le s \le -M,\\-M - 1&\text{if } s < -M - 1.\end{array}\right.
\]
It can be easily checked that $f_M \in C^1(\mathbb{R})$ and $0 \le f'_M(s) \le 1$ for every $s \in \mathbb{R}$. Now, we set $a_M(x,t,s) = a(x,t,f_M(s))$. It is obvious that $a_M$ is of class $C^1$ with respect to the last variable and \eqref{E2.2}--\eqref{E2.4} imply
\begin{align}
	&\frac{\partial a_M}{\partial y}(x,t,y) = \frac{\partial a}{\partial y}(x,t,f_M(y))f'_M(y)\geq \min(0,C_a)\ \forall y\in\mathbb R,\label{E4.1}\\
	&a_M(\cdot,\cdot,0) = a(\cdot,\cdot,0)\in L^{\rd}(0,T;L^{\pd}(\Omega)), \label{E4.2}\\
	&\left|\frac{\partial a_M}{\partial y}(x,t,y)\right|\leq C_{a,M+1}\ \forall y \in \mathbb{R},\label{E4.3}
\end{align}
for almost all $(x,t) \in Q$.

\begin{theorem}
	For any $M > 0$ and all $u \in L^2(Q)$ the equation
	\begin{equation}
		\left\{\begin{array}{l}\displaystyle\frac{\partial y}{\partial t} + A y + a_M(x,t,y) =  u \ \mbox{ in } Q,\\ y = 0  \mbox{ on } \Sigma,\ \ y(0) = y_0 \mbox{ in } \Omega,\end{array}\right.
		\label{E4.4}
	\end{equation}
	has a unique solution $y^M_u \in W(0,T)$. Moreover, $y^M_u$ satisfies the inequalities
	\begin{align}
		&\|y^M_u\|_{C(0,T;L^2(\Omega))} + \|y^M_u\|_{L^2(0,T;H_0^1(\Omega))}\notag\\
		&\hspace{2.4cm}\le K\big(\|u\|_{L^2(Q)} + \|a(\cdot,\cdot,0)\|_{L^2(Q)} + \|y_0\|_{L^2(\Omega)}\big),\label{E4.5}\\
		&\|y^M_u\|_{W(0,T)}\notag\\
		& \le K'\big(\|u\|_{L^2(Q)} + \|y_0\|_{L^2(\Omega)} + \|a(\cdot,\cdot,0)\|_{L^2(Q)} + C_{a,M+1}(M+1)|Q|^{\frac{1}{2}}\big),
		\label{E4.6}
	\end{align}
	where $K$ is the same constant as in \eqref{E2.7} and $K'$ is independent of $M$ and $u$.
	\label{T4.1}
\end{theorem}

\begin{proof}
	From \eqref{E4.3} and the mean value theorem we infer that $|a_M(\cdot,\cdot,s) - a_M(\cdot,\cdot,0)| \le C_{a,M+1}(M+1)$ for all $s \in \mathbb{R}$. Consequently, the estimate
	\[
	\|a_M(\cdot,\cdot,y) - a_M(\cdot,\cdot,0)\|_{L^2(Q)} \le C_{a,M+1}(M+1)|Q|^{\frac{1}{2}}
	\]
	holds. Hence, an easy application of fixed point Schauder's theorem yields the existence of a solution $y^M_u$ in $W(0,T)$. The uniqueness follows in the standard way noting that
	\[
	\int_\Omega[a_M(x,t,y_2) - a_M(x,t,y_1)](y_2 - y_1)\dx \ge \min\{0,C_a\}\|y_2 - y_1\|^2_{L^2(\Omega)}.
	\]
	The proof of the estimate \eqref{E4.5} is the same as the one of \eqref{E2.7}. Inequality \eqref{E4.6} follows from \eqref{E4.5} and the fact that
	\[
	\|a_M(\cdot,\cdot,y)\|_{L^2(Q)} \le \|a_M(\cdot,\cdot,0)\|_{L^2(Q)} + C_{a,M+1}(M+1)|Q|^{\frac{1}{2}}. \hspace{3.25cm}\qed
	\]
\end{proof}

Let us define the mapping $G_M:L^2(Q) \longrightarrow W(0,T)$ associating to every $u$ the corresponding solution $y^M_u$ of \eqref{E4.4}.

\begin{theorem}
	The mapping $G_M$ is of class $C^1$. For all $u, v \in L^2(Q)$ the derivative $z_v = G'_M(u)v$ is the solution of the linearized equation
	\begin{equation}
		\left\{\begin{array}{l}\displaystyle\frac{\partial z}{\partial t} + A z + \frac{\partial a_M}{\partial y}(x,t,y^M_u)z =  v \ \mbox{ in } Q,\\ z = 0  \mbox{ on } \Sigma,\ \ z(0) = 0 \mbox{ in } \Omega,\end{array}\right.
		\label{E4.7}
	\end{equation}
	where $y^M_u = G_M(u)$.
	\label{T4.2}
\end{theorem}

\begin{proof}
	Let us introduce the space
	\[
	Y = \{y \in W(0,T) : \frac{\partial y}{\partial t} + A y \in L^2(Q)\}.
	\]
	This is a Banach space when it is endowed with the graph norm
	\[
	\|y\|_Y = \|y\|_{W(0,T)} + \|\frac{\partial y}{\partial t} + A y\|_{L^2(Q)}.
	\]
	Now, we define the mapping
	\begin{align*}
		&\mathcal{F}_M:Y \times L^2(\Omega) \times L^2(Q) \longrightarrow L^2(Q) \times L^2(\Omega)\\
		&\mathcal{F}_M(y,w,u) = \big(\frac{\partial y}{\partial t} + A y + a_M(\cdot,\cdot,y) - u, y(0) - w\big).
	\end{align*}
	Let us prove that the mapping
	\[
	F_M:W(0,T) \longrightarrow  L^2(Q),\quad F_M(y) = a_M(\cdot,\cdot,y)
	\]
	is of class $C^1$ with
	\[
	DF_M:W(0,T) \longrightarrow \mathcal{L}(W(0,T),L^2(Q)),\quad DF_M(y)z= \frac{\partial a_M}{\partial s}(\cdot,\cdot,y)z.
	\]
	First, we observe that a standard application of a Gagliardo-Nirenberg inequality leads to
	\[
	\|z\|_{L^{\frac{8}{3}}(0,T;L^4(\Omega))} \le C\|z\|^{\frac{1}{4}}_{L^\infty(0,T;L^2(\Omega))}\|z\|^{\frac{3}{4}}_{L^2(0,T;H_0^1(\Omega))} \le C'\|z\|_{W(0,T)}
	\]
	for every $z \in W(0,T)$. Using this inequality, \eqref{E4.3}, and the mean value theorem we infer
	\begin{align*}
		&\|F_M(y + z) - F_M(y) - DF_M(y)z\|^2_{L^2(Q)}\\
		&=\int_Q\Big|a_M(x,t,y(x,t)+z(x,t)) - a_M(x,t,y(x,t)) - \frac{\partial a_M}{\partial s}(x,t,y(x,t))z(x,t)\Big|^2 \dx\dt\\
		&=\int_Q\Big|\frac{\partial a_M}{\partial s}(x,t,y(x,t)+\theta(x,t)z(x,t)) - \frac{\partial a_M}{\partial s}(x,t,y(x,t))\Big|^2z^2(x,t)\dx\dt\\
		&\le\int_0^T\Big\|\frac{\partial a_M}{\partial s}(\cdot,t,y(t)+\theta(t)z(t)) - \frac{\partial a_M}{\partial s}(\cdot,t,y(t))\Big\|_{L^4(\Omega)}^2\|z(t)\|^2_{L^4(\Omega)}\dt\\
		&\le\Big\|\frac{\partial a_M}{\partial s}(\cdot,\cdot,y+\theta z) - \frac{\partial a_M}{\partial s}(\cdot,\cdot,y)\Big\|^2_{L^8(0,T;L^4(\Omega))}\|z\|^2_{L^{\frac{8}{3}}(0,T;L^4(\Omega))}.
	\end{align*}
	From here we deduce
	\[
	\lim_{\|z\|_{W(0,T)} \to 0}\frac{\|F_M(y + z) - F_M(y) - DF_M(y)z\|_{L^2(Q)}}{\|z\|_{W(0,T)}} = 0.
	\]
	Hence, $F_M$ is Fr\'echet differentiable. The continuity of $DF_M$ is immediate and, consequently, $F_M$ is of class $C^1$. Using this and the continuity of the embedding $Y \subset W(0,T) \subset C([0,T];L^2(\Omega))$, we conclude that $\mathcal{F}_M$ is of class $C^1$. Moreover, we have $\mathcal{F}_M(y^M_u,y_0,u) = (0,0)$. An easy application of the implicit function theorem proves Theorem \ref{T4.2}.
	\qed\end{proof}

For every $M > 0$ we consider the control problems
\[
\PbM \quad  \inf_{u \in \uad \cap L^2(Q)} J_M(u):= \frac{1}{2}\int_Q (y^M_u(x,t) - y_d(x,t))^2\dx\dt + \frac{\kappa}{2}\int_Q u(x,t)^2\dx\dt,
\]
where $y^M_u$ denotes the solution of \eqref{E4.4}. Problem \PbM has at least a solution $u_M$. This is consequence of the coercivity of $J_M$ on $L^2(Q)$, the fact that $\uad \cap L^2(Q)$ is closed and convex in $L^2(Q)$, and the lower semicontinuity of $J_M$ with respect to the weak topology of $L^2(Q)$. The last statement follows easily from the estimate \eqref{E4.6} and the compactness of the embedding $W(0,T) \subset L^2(Q)$.

From the chain rule and Theorem \ref{T4.2} we infer that $J_M : L^2(Q) \longrightarrow \mathbb{R}$ is of class $C^1$ and its derivative is given by the expression
\begin{equation}
	J'_M(u)v = \int_Q(\varphi + \kappa u)v\dx\dt,
	\label{E4.8}
\end{equation}
where $\varphi \in W(0,T)$ is the solution of the adjoint state equation
\begin{equation}
	\left\{\begin{array}{l}\displaystyle-\frac{\partial\varphi}{\partial t} + A^*\varphi + \frac{\partial a_M}{\partial y}(x,t,y^M_u)\varphi =  y^M_u - y_d\ \mbox{ in } Q,\\ \varphi = 0  \mbox{ on } \Sigma,\ \ \varphi(T) = 0 \mbox{ in } \Omega.\end{array}\right.
	\label{E4.9}
\end{equation}

\begin{theorem}
	Let $u_M$ be a solution of \PbM. Then, there exist functions $y_M, \varphi_M \in W(0,T)$ and $\mu_M \in L^2(Q)$ such that
	\begin{align}
		&\left\{\begin{array}{l}\displaystyle\frac{\partial y_M}{\partial t} + A y_M + a_M(x,t,y_M) =  u_M \ \mbox{ in } Q,\\ y_M = 0  \mbox{ on } \Sigma,\ \ y_M(0) = y_0 \mbox{ in } \Omega,\end{array}\right.
		\label{E4.10}\\
		&\left\{\begin{array}{l}\displaystyle-\frac{\partial\varphi_M}{\partial t} + A^*\varphi_M + \frac{\partial a_M}{\partial y}(x,t,y_M)\varphi_M =  y_M - y_d\ \mbox{ in } Q,\\ \varphi_M = 0  \mbox{ on } \Sigma,\ \ \varphi_M(T) = 0 \mbox{ in } \Omega,\end{array}\right.
		\label{E4.11}\\
		&\int_Q\mu_M(x,t)(u(x,t) - u_M(x,t))\dx\dt \le 0\quad \forall u \in \uad \tcr{\cap L^2(Q)},\label{E4.12}\\
		&\varphi_M + \kappa u_M + \mu_M = 0.\label{E4.13}
	\end{align}
	\label{T4.3}
\end{theorem}

The proof of this theorem is the same as the one of Theorem \ref{T3.1}.

\begin{theorem}
	Let $(u_M,y_M,\varphi_M,\mu_M)$ be as in Theorem \ref{T4.3}. Then, there exists a constant $K_\infty > 0$ such that
	\begin{equation}
		\|(u_M,y_M,\varphi_M,\mu_M)\|_{L^\infty(Q)^4} \le K_\infty \quad \forall M > 0.
		\label{E4.14}
	\end{equation}
	\label{T4.4}
\end{theorem}

\begin{proof}
	As in the proof for the first statement of \eqref{E3.8}, we have that \eqref{E4.12} and \eqref{E4.13} yield $|\mu_M(x,t)||u_M(x,t)| = \mu_M(x,t)u_M(x,t)$ for almost all $(x,t) \in Q$.
	
	We denote by $y_M^0$ the solution of \eqref{E4.4} associated with the control identically zero. Then, according to Theorem \ref{T4.1}, inequality \eqref{E4.5} implies that
	\[
	\|y_M^0\|_{C(0,T;L^2(\Omega))} \le K\big(\|a(\cdot,\cdot,0)\|_{L^2(Q)} + \|y_0\|_{L^2(\Omega)}\big)\quad \forall M > 0.
	\]
	From this inequality we infer
	\[
	\|y^0_M\|_{L^2(Q)} \le C_1 = \sqrt{T}K\big(\|a(\cdot,\cdot,0)\|_{L^2(Q)} + \|y_0\|_{L^2(\Omega)}\big) \quad \forall M > 0.
	\]
	Since $u_M$ is solution of \PbM and $u \equiv 0$ is an admissible control for \PbM we get
	\[
	\frac{\kappa}{2}\|u_M\|^2_{L^2(Q)} \le J_M(u_M) \le J_M(0) = \frac{1}{2}\|y^0_M - y_d\|^2_{L^2(Q)}.
	\]
	This leads to
	\[
	\|u_M\|_{L^2(Q)} \le \frac{1}{\sqrt{\kappa}}\|y^0_M - y_d\|_{L^2(Q)} \le C_2 = \frac{1}{\sqrt{\kappa}}(C_1 + \|y_d\|_{L^2(Q)}) \quad \forall M > 0.
	\]
	Using again \eqref{E4.5} and this estimate we deduce
	\[
	\|y_M\|_{L^\infty(0,T;L^2(\Omega))} \le C_3 = K_2\big(C_2 + \|a(\cdot,\cdot,0)\|_{L^2(Q)} + \|b(\cdot,\cdot,0)\|_{L^2(\Sigma)} + \|y_0\|_{L^2(\Omega)}\big) \quad \forall M > 0.
	\]
	Using this estimate we can infer the boundedness of $\varphi_M$ by a constant independent of $M$. The idea of the proof is to make the substitution $\varphi_M(x,t) = {\rm e}^{-|C_a|t}\psi_M(x,t)$, where $C_a$ is given in \eqref{E2.2}. Then, $\psi$ satisfies the equation
	\[
	\left\{\begin{array}{l}\displaystyle-\frac{\partial\psi_M}{\partial t} + A^*\psi_M + \big(\frac{\partial a_M}{\partial y}(x,t,y_M) + |C_a|\big)\psi_M =  {\rm e}^{|C_a|t}(y_M - y_d)\ \mbox{ in } Q,\\ \psi_M = 0  \mbox{ on } \Sigma,\ \ \psi_M(T) = 0 \mbox{ in } \Omega.\end{array}\right.
	\]
	Since \eqref{E4.1} implies that $\frac{\partial a_M}{\partial y}(x,t,y^M_u) + |C_a| \ge 0$, we  apply \cite[Theorem III-7.1]{Lad-Sol-Ura68} to deduce the existence of a constant $C >0 $ independent of $M$ such that
	\begin{align*}
		&\|\psi_M\|_{L^\infty(Q)} \le C\big({\rm e}^{|C_a|T}\big[\|y_M\|_{L^\infty(0,T;L^2(\Omega))} + \|y_d\|_{L^\rd(0,T;L^\pd(\Omega))}\big]\big)\\
		&\le C_4 = C\big({\rm e}^{|C_a|T}\big[C_3 + \|y_d\|_{L^\rd(0,T;L^\pd(\Omega))}\big]\big)\quad \forall M > 0.
	\end{align*}
	From here we infer the estimate $\|\varphi_M\|_{L^\infty(Q)} \le \|\psi_M\|_{L^\infty(Q)} \le C_4$ for every $M > 0$. Now, using that $u_M$ and $\mu_M$ have the same sign almost everywhere in $Q$, we deduce from \eqref{E4.13}
	\[
	\kappa|u_M(x,t)| \le |\kappa u_M(x,t) + \mu_M(x,t)| = |\varphi_M(x,t)| \le C_4,
	\]
	which proves that $\|u_M\|_{L^\infty(Q)} \le \frac{C_4}{\kappa}$ for every $M > 0$. Moreover, the bounds from $u_M$ and $\varphi_M$ along with \eqref{E4.13} imply that $\|\mu_M\|_{L^\infty(Q)} \le 2C_4$. Finally, the estimate of $y_M$ in $L^\infty(Q)$ independently of $M$ follows from \eqref{E4.10}, Theorem \ref{T2.1}, and the estimate for $u_M$.
	\qed\end{proof}

\begin{remark}
	The assumption $\kappa > 0$ was used in an essential manner in the above proof.
	\label{R4.1}
\end{remark}

\begin{theorem}
	Let $M \ge K_\infty$ be arbitrary, where $K_\infty$ satisfies \eqref{E4.14}. Let $u_M$ be a solution of \PbM. Then, $u_M$ is a solution of \Pb.
	\label{T4.5}
\end{theorem}

\begin{proof}
	First we observe that $\|y_M\|_{L^\infty(Q)} \le M$ and hence $a_M(x,t,y_M) = a(x,t,y_M)$. Therefore, $y_M$ is the solution of \eqref{E1.1} corresponding to $u_M$ and, consequently, $J_M(u_M) = J(u_M)$.
	
	Given $u \in \uad \cap L^\infty(Q)$ arbitrary, let $y_u$ be the associated solution of \eqref{E1.1} and set $M_0 = \|y_u\|_{L^\infty(Q)}$. If $M_0 \le M$, then it is obvious that $a_M(x,t,y_u) = a(x,t,y_u)$ and, hence, $J_M(u) = J(u)$. Therefore, the optimality of $u_M$ implies $J(u_M) = J_M(u_M) \le J_M(u) = J(u)$.
	
	If $M_0 > M$, we take a solution $u_{M_0}$ of (P$_{M_0}$). Then, Theorem \ref{T4.4} implies that the solution $y_{M_0}$ of \eqref{E4.10} with $M$ replaced by $M_0$ satisfies $\|y_{M_0}\|_{L^\infty(Q)} \le M$ and, consequently, $a_{M_0}(x,t,y_{M_0}) = a_M(x,t,y_{M_0}) = a(x,t,y_{M_0})$ and $J_{M_0}(u_{M_0}) = J_M(u_{M_0}) = J(u_{M_0})$.
	These facts along with the optimality of $u_M$ and $u_{M_0}$ lead to
	\[
	J(u_M) = J_M(u_M) \le J_M(u_{M_0}) = J_{M_0}(u_{M_0}) \le J_{M_0}(u) = J(u),
	\]
	which proves that $u_M$ is a solution of \Pb.
	\qed\end{proof}

\begin{remark}
	Let us compare problem \Pb with the control problems
	\[
	\Pbr \quad  \inf_{u \in \uad \cap L^r(0,T;L^2(\Omega))} J(u):= \frac{1}{2}\int_Q (y_u(x,t) - y_d(x))^2\dx\dt + \frac{\kappa}{2}\int_Q u(x,t)^2\dx\dt,
	\]
	where $r \in (\frac{4}{4 - n},\infty)$. We observe that Theorems \ref{T2.1} and \ref{T2.2}, and Corollary \ref{C2.1} are applicable to deduce that any solution of \Pbr satisfies the optimality conditions \eqref{E3.2}--\eqref{E3.5}. Then, the arguments of Theorem \ref{T4.4} apply to deduce that any solution of \Pbr belongs to $L^\infty(Q)$. Let us check that problems \Pb and \Pbr are equivalent in the sense that both have the same solutions. Indeed, since $\uad \cap L^r(0,T;L^2(\Omega)) \supset \uad \cap L^\infty(Q)$, it is obvious that every solution of \Pbr is a solution of \Pb. Conversely, let $\bar u$ be a solution of \Pb and take $u \in \uad \cap L^r(0,T;L^2(\Omega))$ arbitrarily. For every integer $k \ge 1$ we set $u_k = \proj_{[-k,+k]}(u)$. Then, it is obvious that $u_k \in \uad \cap L^\infty(Q)$ and $u_k \to u$ in $L^r(0,T;L^2(\Omega))$. Using the optimality of $\bar u$ we have $J(\bar u) \le J(u_k)$ for all $k$, and passing to the limit we infer that $J(\bar u) \le J(u)$. Since $u$ was arbitrary, this implies that $\bar u$ is a solution of \Pbr.
	\label{R4.2}
\end{remark}

\section{Second Order Optimality Conditions}
\label{S5}

We consider the Lipschitz and convex mapping $j:L^1(\Omega) \longrightarrow \mathbb{R}$ defined by $j(v) = \|v\|_{L^1(\Omega)}$. Its directional derivative is given by the expression
\begin{equation}
	j'(u;v) = \int_{\Omega^+_u}v(x)\dx - \int_{\Omega^-_u}v(x)\dx + \int_{\Omega^0_u}|v(x)|\dx \quad \forall u, v \in L^1(\Omega),
	\label{E5.1}
\end{equation}
where
\[
\Omega^+_u =\{x \in \Omega : u(x) > 0\}, \ \Omega^-_u =\{x \in \Omega : u(x) < 0\} \text{ and } \Omega_u^0 = \Omega \setminus (\Omega_u^+ \cup \Omega_u^-).
\]

In order to derive the second order optimality conditions for \Pb, we define the cone of critical directions. For a control $\bar u \in \uad \cap L^\infty(Q)$ satisfying the first order optimality conditions \eqref{E3.2}--\eqref{E3.5} we set
\[
C_{\bar u} = \Big\{v \in L^2(Q) : J'(\bar u)v = 0 \text{ and } j'(\bar u(t);v(t)) \left\{\begin{array}{cl} = 0&\text{if } t \in I_\gamma^+,\\\le 0&\text{if } t \in I_\gamma\setminus I^+_\gamma,\end{array}\right.\Big\},
\]
where
\[
I_\gamma = \{t \in (0,T) : j(\bar u(t)) = \gamma\} \ \ \text{ and }\ \ I^+_\gamma = \{t \in I_\gamma : \bar\mu(t) \not\equiv 0 \text{ in } \Omega\}.
\]

We first prove the second order necessary conditions. Given an element $v \in C_{\bar u}$, the classical approach to prove these second order conditions consists of taking a sequence $\{v_k\}_{k = 1}^\infty$ converging to $v$ such that $\bar u + \rho v_k$ is a feasible control for \Pb for every $\rho > 0$ small enough. The way of taking this sequence is different from the case where  box control constraints are considered. The main reason for this difference is that the functional $j$, defining the constraint, is not differentiable and that it is non-local in space. Even the approach followed in the case where $j$ is involved in the cost functional cannot be used in our framework; see \cite{CHW2017}. The proof makes an essential use of the following lemma.

\begin{lemma}
	Let $v \in L^2(Q)$ satisfy $j'(\bar u(t);v(t))  = 0$ for almost all $t \in I_\gamma^+$. Then,  $J'(\bar u)v = 0$ holds if and only if
	\begin{equation}
		\|\bar\mu(t)\|_{L^\infty(\Omega)}|v(x,t)| = \bar\mu(x,t)v(x,t)\text{ for a.a. } (x,t) \in \Omega_{\bar u(t)}^0 \times I^+_\gamma.
		\label{E5.2}
	\end{equation}
	As a consequence, every element $v$ of $C_{\bar u}$ satisfies \eqref{E5.2}.
	\label{L5.1}
\end{lemma}

\begin{proof}
	From \eqref{E2.20}, \eqref{E3.5}, and \eqref{E3.8} we infer
	\begin{align*}
		&J'(\bar u)v = \int_Q(\bar\varphi + \kappa\bar u)v\dx \dt = -\int_Q\bar\mu v\dx \dt = -\int_{I_\gamma^+}\int_\Omega\bar\mu v\dx \dt\\
		&= -\int_{I^+_\gamma}\|\bar\mu(t)\|_{L^\infty(\Omega)}\left\{\int_{\Omega_{\bar u(t)}^+}v \dx - \int_{\Omega_{\bar u(t)}^-}v \dx\right\} - \int_{I^+_\gamma}\int_{\Omega^0_{\bar u(t)}}\bar\mu v\dx \dt.
	\end{align*}
	Using that $j'(\bar u(t);v(t))  = 0$ for almost all $t \in I_\gamma^+$ and \eqref{E5.1} we get
	\[
	\int_{\Omega_{\bar u(t)}^+}v \dx - \int_{\Omega_{\bar u(t)}^-}v \dx = -\int_{\Omega^0_{\bar u(t)}}|v|dx.
	\]
	Inserting this in the previous identity we obtain
	\[
	J'(\bar u)v = \int_{I_\gamma^+}\int_{\Omega^0_{\bar u(t)}}[\|\mu(t)\|_{L^\infty(\Omega)}|v| - \bar\mu v]\dx \dt.
	\]
	Since $\bar\mu v \le \|\mu(t)\|_{L^\infty(\Omega)}|v|$, we deduce from the above equality that $J'(\bar u)v = 0$ if and only if \eqref{E5.2} holds.
\end{proof}

\begin{theorem}
	Let $\bar u$ be a local solution of \Pb in the $L^r(0,T;L^2(\Omega))$ sense  with $r > \frac{4}{4-n}$. Then, the inequality $J''(\bar u)v^2 \ge 0$ holds for all $v \in C_{\bar u}$.
	\label{T5.1}
\end{theorem}

\begin{proof}
	Let $v$ be an element of $C_{\bar u} \cap L^\infty(0,T;L^2(\Omega))$. We will prove that $J''(\bar u)v^2 \ge 0$. Later, we will remove the assumption $v \in L^\infty(0,T;L^2(\Omega))$. Set
	\[
	g(x,t) = \left\{\begin{array}{cl}\displaystyle\frac{v(x,t)}{|\bar u(x,t)|}&\text{if } x \not\in \Omega_{\bar u(t)}^0,\vspace{2mm}\\ 0 &\text{otherwise,}\end{array}\right.\quad \text{ and }\quad a(t) = \int_\Omega g(x,t)\bar u(x,t)\dx.
	\]
	From \eqref{E5.1} we infer
	\[
	j'(\bar u(t);v(t)) = a(t) + \int_{\Omega_{\bar u(t)}^0}|v(x,t)|\dx.
	\]
	For every integer $k \ge 1$ we put
	\begin{align*}
		&a_k(t) = \int_\Omega\proj_{[-k,+k]}(g(x,t))\bar u(x,t)\dx,\\
		&g_k(x,t) = \proj_{[-k,+k]}(g(x,t))|\bar u(x,t)| + \frac{a(t) - a_k(t)}{\gamma}\bar u(x,t),\vspace{2mm}\\
		&v_k(x,t) = \left\{\begin{array}{cl}0&\text{if } \gamma - \frac{1}{k} < \|\bar u(t)\|_{L^1(\Omega)} < \gamma,\vspace{2mm}\\g_k(x,t) + v(x,t)\chi_{\Omega_{\bar u(t)}^0}(x)&\text{if } \|\bar u(t)\|_{L^1(\Omega)} = \gamma,\\ v(x,t) & \text{otherwise,}\end{array}\right.
	\end{align*}
	where $\chi_{\Omega_{\bar u(t)}^0}(x)$ takes the value $1$ if $x \in \Omega_{\bar u(t)}^0$ and $0$ otherwise.
	
	Using that $|\proj_{[-k,+k]}(g(x,t))\bar u(x,t)| \le |v(x,t)|$ and the pointwise almost everywhere convergence $\proj_{[-k,+k]}(g(x,t))\bar u(x,t) \to g(x,t)\bar u(x,t)$  in $Q$, we deduce with Lebesgue's Theorem that $\lim_{k \to \infty}a_k(t) = a(t)$ for almost all $t \in (0,T)$. Therefore, we have that $v_k(x,t) \to v(x,t)$ for almost all $(x,t) \in Q$. Moreover, we have
	\[
	|g_k(x,t)| \le |v(x,t)| + \frac{2}{\gamma}\|v\|_{L^\infty(0,T;L^1(\Omega))}\|\bar u\|_{L^\infty(Q)}
	\]
	and
	\[
	|v_k(x,t)| \le |v(x,t)| + \frac{2}{\gamma}\|v\|_{L^\infty(0,T;L^1(\Omega))}\|\bar u\|_{L^\infty(Q)}
	\text{ for a.a. } (x,t) \in Q.
	\]
	Once again, with Lebesgue's Theorem we get $v_k \to v$ in $L^r(0,T;L^2(\Omega))$ for every $r < \infty$.
	
	Let us prove that $J'(\bar u)v_k = 0$. To this end, we apply Lemma \ref{L5.1}. Actually, we are going to prove that $v_k \in C_{\bar u}$. Given $t \in I_\gamma$, taking into account \eqref{E5.1} and the fact that $j(\bar u(t)) = \|\bar u(t)\|_{L^1(\Omega)} = \gamma$ we get
	\begin{align*}
		&j'(\bar u(t);v_k(t))\\
		&= \int_{\Omega^+_{\bar u(t)}}\proj_{[-k,+k]}(g(x,t))|\bar u(x,t)|\dx - \int_{\Omega^-_{\bar u(t)}}\proj_{[-k,+k]}(g(x,t))|\bar u(x,t)|\dx\\
		&+ \frac{a(t)-a_k(t)}{\gamma}\Big[\int_{\Omega^+_{\bar u(t)}}\bar u(x,t)\dx - \int_{\Omega^-_{\bar u(t)}}\bar u(x,t)\dx\Big] + \int_{\Omega_{\bar u(t)}^0}|v(x,t)|\dx\\
		&= \int_\Omega\proj_{[-k,+k]}(g(x,t))\bar u(x,t)\dx + \frac{a(t)-a_k(t)}{\gamma}j(\bar u(t)) + \int_{\Omega_{\bar u(t)}^0}|v(x,t)|\dx\\
		&= a(t) + \int_{\Omega_{\bar u(t)}^0}|v(x,t)|\dx = j'(\bar u(t),v(t)) \left\{\begin{array}{cl}= 0&\text{if } t \in I^+_\gamma,\\\le 0&\text{if } t \in I_\gamma\setminus I^+_\gamma,\end{array}\right.
	\end{align*}
	where we used that $v \in C_{\bar u}$ in the last step.
	
	In the case where $\|\bar u(t)\|_{L^1(\Omega)} < \gamma$, according to the definition of $v_k$, we have that $v_k(x,t)$ is equal to 0 or to $v(x,t)$. Since $v$ satisfies \eqref{E5.2} due to the fact that $v \in C_{\bar u}$, we deduce that $v_k$ also satisfies \eqref{E5.2}. Then, Lemma \ref{L5.1} implies that $J'(\bar u)v_k = 0$. Therefore, $v_k \in C_{\bar u}$ holds.
	
	Take $\rho_k > 0$ such that
	\[
	\rho_k\big(k + \frac{2}{\gamma}\|v\|_{L^\infty(0,T;L^1(\Omega))}\big) < \frac{1}{k\max\{1,\gamma\}}.
	\]
	Then, we have for each fixed $k$ and $\forall \rho \in (0,\rho_k)$
	\[
	\rho\big(|\proj_{[-k,+k]}(g(x,t))| + \big|\frac{|a(t) - a_k(t)|}{\gamma}\big|\big) \le \rho\big(k + \frac{2}{\gamma}\|v\|_{L^\infty(0,T;L^1(\Omega))}\big) < \frac{1}{k}.
	\]
	Using this estimate we have that $\|\bar u(t) + \rho v_k(t)\| \le \gamma$ if $j(\bar u(t)) = \gamma$ and $0 < \rho < \rho_k$:
	\begin{align*}
		&\|\bar u(t) + \rho v_k(t)\|_{L^1(\Omega)}\\
		&= \int_{\Omega\setminus\Omega^0_{\bar u(t)}}\Big|\bar u(t)[1 + \rho\big[\proj_{[-k,+k]}(g(x,t))\sign(\bar u(x,t)) + \frac{a(t) - a_k(t)}{\gamma}\big]\Big|\dx\\
		& + \rho\int_{\Omega^0_{\bar u(t)}}|v(x,t)|\dx\\
		&= \int_{\Omega\setminus\Omega^0_{\bar u(t)}}|\bar u(t)|[1 + \rho\big[\proj_{[-k,+k]}(g(x,t))\sign(\bar u(x,t)) + \frac{a(t) - a_k(t)}{\gamma}\big]\dx\\
		& + \rho\int_{\Omega^0_{\bar u(t)}}|v(x,t)|\dx\\
		&= \int_\Omega|\bar u(t)|\dx + \rho\left\{\int_\Omega\big[\proj_{[-k,+k]}(g(x,t))\bar u(x,t) + \frac{a(t) - a_k(t)}{\gamma}|\bar u(x,t)|\big]\dx\right.\\
		&\left. + \int_{\Omega^0_{\bar u(t)}}|v(x,t)|\dx\right\}\\
		&=\|\bar u(t)\|_{L^1(\Omega)} + \rho\left\{a(t) + \int_{\Omega^0_{\bar u(t)}}|v(x,t)|\dx\right\} = \gamma + \rho j'(\bar u(t);v(t)) \le \gamma.
	\end{align*}
	
	In the case $\gamma - \frac{1}{k} < \|\bar u(t)\|_{L^1(\Omega)} < \gamma$, we have that $v_k(t) = 0$ and, consequently
	\[
	\|\bar u(t) + \rho v_k(t)\|_{L^1(\Omega)} = \|\bar u(t)\|_{L^1(\Omega)} < \gamma.
	\]
	If $\|\bar u(t)\|_{L^1(\Omega)} < \gamma - \frac{1}{k}$, then we get
	\[
	\|\bar u(t) + \rho v_k(t)\|_{L^1(\Omega)} \le \gamma - \frac{1}{k} + \rho\|v\|_{L^\infty(0,T;L^1(\Omega))} < \gamma.
	\]
	
	Using the local optimality of $\bar u$, the fact that $\bar u + \rho v_k \in \uad$, $J'(\bar u)v_k = 0$, and making a Taylor expansion we get for every $\rho < \rho_k$ small enough
	\[
	0 \le J(\bar u + \rho v_k) - J(\bar u) = \rho J'(\bar u)v_k + \frac{\rho^2}{2}J''(\bar u + \theta\rho v_k)v_k^2 = \frac{\rho^2}{2}J''(\bar u + \theta\rho v_k)v_k^2.
	\]
	Dividing the above inequality by $\rho^2/2$ and making $\rho \to 0$ we obtain with Corollary \ref{C2.1} that $J''(\bar u)v_k^2 \ge 0$. Since $v_k \to v$ in $L^2(Q)$, we pass to the limit when $k \to \infty$ and conclude that $J''(\bar u)v^2 \ge 0$.
	
	Finally, we take $v \in C_{\bar u}$ arbitrary and for every $k \ge 1$ set
	\[
	v_k(x,t) = \frac{v(x,t)}{1 + \frac{1}{k}\|v(t)\|_{L^1(\Omega)}}.
	\]
	Then, we have
	\begin{align*}
		&J'(\bar u)v_k = \frac{1}{1 + \frac{1}{k}\|v(t)\|_{L^1(\Omega)}}J'(\bar u)v = 0\ \text{ and }\\
		&j'(\bar u(t);v_k(t)) = \frac{1}{1 + \frac{1}{k}\|v(t)\|_{L^1(\Omega)}}j'(\bar u(t);v(t)) \left\{\begin{array}{cl} = 0&\text{if } t \in I_\gamma^+,\\\le 0&\text{if } t \in I_\gamma\setminus I^+_\gamma.\end{array}\right.
	\end{align*}
	Therefore, $v_k \in C_{\bar u} \cap L^\infty(0,T;L^1(\Omega))$ and $v_k \to v$ in $L^2(Q)$ is satisfied. Hence, we get $J''(\bar u)v^2 = \lim_{k \to \infty}J''(\bar u)v_k^2 \ge 0$, which concludes the proof.
	\qed\end{proof}

\begin{theorem}
	Let $\bar u \in \uad \cap L^\infty(Q)$ satisfy the first order optimality conditions \eqref{E3.2}--\eqref{E3.5}. If $J''(\bar u)v^2 > 0$ $\forall v \in C_{\bar u} \setminus \{0\}$ holds, then for each $r \in \big(\frac{4}{4 - n},\infty]$ there exist $\delta > 0$ and $\varepsilon > 0$ such that
	\begin{equation}
		J(\bar u) + \frac{\delta}{2}\|u - \bar u\|^2_{L^2(Q)} \le J(u)\quad \forall u \in \uad \cap B_\varepsilon(\bar u),
		\label{E5.3}
	\end{equation}
	where $B_\varepsilon(\bar u) = \{u \in L^r(0,T;L^2(\Omega)) : \|u - \bar u\|_{L^r(0,T;L^2(\Omega))} \le \varepsilon\}$.
	\label{T5.2}
\end{theorem}

\begin{proof}
	We proceed by contradiction. If \eqref{E5.3} is false for every $\delta > 0$ and $\varepsilon > 0$, then for every integer $k \ge 1$ there exists an element $u_k \in \uad$ such that
	\begin{equation}
		\|u_k - \bar u\|_{L^r(0,T;L^2(\Omega))} < \frac{1}{k}\ \text{ and } \ J(u_k) < J(\bar u) + \frac{1}{2k}\|u_k - \bar u\|^2_{L^2(Q)}.
		\label{E5.4}
	\end{equation}
	Let us set $\rho_k = \|u_k - \bar u\|_{L^2(Q)}$ and $v_k = (u_k - \bar u)/\rho_k$. Then, we have  $\|v_k\|_{L^2(Q)} = 1$ and, taking a subsequence that we denote in the same way, we have $v_k \rightharpoonup v$ in $L^2(Q)$. We divide the proof in several steps.
	
	{\em Step I - $J'(\bar u)v = 0$.} From \eqref{E3.4} and \eqref{E3.5} we infer that $J'(\bar u)(u_k - \bar u) \ge 0$ for every $k \ge 1$. Therefore, $J'(\bar u)v_k \ge 0$ and passing to the limit we obtain $J'(\bar u)v \ge 0$. Now, using \eqref{E5.4} along with the mean value theorem we get for some $\theta_k \in (0,1)$
	\[
	J(u_k) - J(\bar u) = J'(\bar u + \theta_k(u_k - \bar u))(u_k - \bar u) < \frac{1}{2k}\|u_k - \bar u\|^2_{L^2(Q)}.
	\]
	Dividing this inequality by $\rho_k$ we obtain
	\[
	J'(\bar u + \theta_k(u_k - \bar u))v_k < \frac{1}{2k}\|u_k - \bar u\|_{L^2(Q)}.
	\]
	Then, passing to the limit when $k \to \infty$ it follows $J'(\bar u)v \le 0$.
	
	{\em Step II - $v \in C_{\bar u}$.} Since $\bar u(t) + \lambda v_k(t) = \bar u(t) + \frac{\lambda}{\rho_k}(u_k(t) - \bar u(t)) \in \uad$ for every $0 < \lambda < \rho_k$, we get for almost every $t \in I_\gamma$
	\begin{align*}
		j'(\bar u(t);v_k(t)) &= \lim_{\lambda \searrow 0}\frac{\|\bar u(t) + \lambda v_k(t)\|_{L^1(\Omega)} - \|\bar u(t)\|_{L^1(\Omega)}}{\lambda}\\
		&= \lim_{\lambda \searrow 0}\frac{\|\bar u(t) + \lambda v_k(t)\|_{L^1(\Omega)} - \gamma}{\lambda} \le 0.
	\end{align*}
	Take a measurable subset $J \subset I_\gamma$. Since the functional
	\[
	u \in L^2(Q) \longrightarrow \int_Jj'(\bar u(t);u(t))\dt \in \mathbb{R}
	\]
	is continuous and convex, recall \eqref{E5.1}, the weak convergence $v_k \rightharpoonup v$ in $L^2(Q)$ implies
	\[
	\int_Jj'(\bar u(t);v(t))\dt \le \liminf_{k \to \infty}\int_Jj'(\bar u(t);v_k(t))\dt \le 0.
	\]
	Since $J \subset I_\gamma$ is an arbitrary measurable set, we infer for almost all $t \in I_\gamma$
	\begin{equation}
		\int_{\Omega^+_{\bar u(t)}}v(t) \dx - \int_{\Omega^-_{\bar u(t)}}v(t) \dx +\int_{\Omega^0_{\bar u(t)}}|v(t)| \dx = j'(\bar u(t);v(t)) \le 0.
		\label{E5.5}
	\end{equation}
	Identities \eqref{E3.5} and $J'(\bar u)v = 0$, and \eqref{E3.8} imply
	\begin{align}
		&0 = \int_Q\bar\mu(x,t) v(x,t)\dx \dt = \int_{I^+_\gamma}\int_\Omega \bar\mu(x,t) v(x,t)\dx \dt\notag\\
		&=\int_{I^+_\gamma}\left\{\|\bar\mu(t)\|_{L^\infty(\Omega)}\Big[\int_{\Omega^+_{\bar u(t)}}v(t) \dx - \int_{\Omega^-_{\bar u(t)}}v(t) \dx\Big] + \int_{\Omega^0_{\bar u(t)}}\bar\mu(t)v(t) \dx\right\}\dt.
		\label{E5.6}
	\end{align}
	From \eqref{E5.5} we deduce
	\[
	\int_{I^+_\gamma}\left\{\|\bar\mu(t)\|_{L^\infty(\Omega)}\Big[\int_{\Omega^+_{\bar u(t)}}v(t) \dx - \int_{\Omega^-_{\bar u(t)}}v(t) \dx +\int_{\Omega^0_{\bar u(t)}}|v(t)| \dx\Big]\right\}\dt \le 0.
	\]
	The last two relations lead to
	\[
	\int_{I^+_\gamma}\left\{\int_{\Omega^0_{\bar u(t)}}\Big[\|\bar\mu(t)\|_{L^\infty(\Omega)}|v(t)| - \bar\mu(t)v(t)\Big]\dx\right\}\dt \le 0.
	\]
	This is possible if and only if $\|\bar\mu(t)\|_{L^\infty(\Omega)}|v(x,t)| = \bar\mu(x,t)v(x,t)$ for almost all $t \in I^+_\gamma$ and $x \in \Omega^0_{\bar u(t)}$. Inserting this identity in \eqref{E5.6} we get
	\[
	0 = \int_Q\bar\mu(x,t) v(x,t)\dx \dt = \int_{I^+_\gamma}\|\bar\mu(t)\|_{L^\infty(\Omega)}j'(\bar u(t);v(t))\dt.
	\]
	Finally, this identity and \eqref{E5.5} yield $j'(\bar u(t);v(t)) = 0$ for almost all $t \in I^+_\gamma$. Therefore, we conclude with {\em Step I} that $v \in C_{\bar u}$.
	
	{\em Step III - $J''(\bar u)v^2 \le 0$.} From \eqref{E5.4} and a Taylor expansion we infer
	\[
	\rho_kJ'(\bar u)v_k + \frac{\rho_k^2}{2}J''(\bar u + \theta_k\rho_kv_k)v_k^2 = J(u_k) - J(\bar u) < \frac{1}{2k}\|u_k - \bar u\|^2_{L^2(Q)}.
	\]
	Since $J'(\bar u)v_k = \frac{1}{\rho_k}J'(\bar u)(u_k - \bar u) \ge 0$, we deduce from the above inequality
	\begin{equation}
		J''(\bar u + \theta_k(u_k - \bar u))v_k^2 = J''(\bar u + \theta_k\rho_kv_k)v_k^2 < \frac{1}{k}.
		\label{E5.7}
	\end{equation}
	The strong convergence $\bar u + \theta_k(u_k - \bar u) \to \bar u$ in $L^r(0,T;L^2(\Omega))$ yields the uniform convergences $y_{\theta_k} \to \bar y$ and $\varphi_{\theta_k} \to \bar\varphi$ in $L^\infty(Q)$, where $y_{\theta_k}$ and $\varphi_{\theta_k}$ are the state and adjoint state associated with $\bar u + \theta_k(u_k - \bar u)$. This also implies that $z_{\theta_k,v_k} \to z_v$ strongly in $L^2(Q)$, where $z_v$ is the solution of \eqref{E2.20} for $y_u = \bar y$ and $z^2_{\theta_k,v_k}$ is the solution of \eqref{E2.20} with $v = v_k$ and $y_u = y_{\theta_k}$. Then, we can pass to the limit in \eqref{E5.7} when $k \to \infty$ and deduce that $J''(\bar u)v^2 \le 0$.
	
	{\em Step IV - Final contradiction.} Since $v \in C_{\bar u}$ and $J''(\bar u)v^2 \le 0$, according to the assumptions of the theorem, this is only possible if $v = 0$. Therefore, we have that $v_k \rightharpoonup 0$ in $L^2(Q)$ and, consequently, $z^2_{\theta_k,v_k} \to 0$ strongly in $L^2(Q)$.  Now, using that $\|v_k\|_{L^2(Q)} = 1$ and \eqref{E2.21}, we infer from \eqref{E5.7}
	\begin{align*}
		&0 \ge \liminf_{k \to \infty}J''(\bar u + \theta_k(u_k - \bar u))v_k^2\\
		& = \liminf_{k \to \infty}\int_Q\Big[\big(1 - \frac{\partial^2a}{\partial y^2}(x,t,y_{\theta_k})\varphi_{\theta_k}\big)z^2_{\theta_k,v_k} + \kappa v_k^2\Big]\dx\dt\\
		&\lim_{k \to \infty}\int_Q\big(1 - \frac{\partial^2a}{\partial y^2}(x,t,y_{\theta_k})\varphi_{\theta_k}\big)z^2_{\theta_k,v_k}\dx\dt + \kappa = \kappa,
	\end{align*}
	which contradicts our assumption $\kappa > 0$.
	\qed\end{proof}

The next theorem establishes that the sufficient condition for local optimality, $J''(\bar u)v^2 > 0$ for every $v \in C_{\bar u} \setminus \{0\}$, provides a useful tool for the numerical analysis of the control problem. Given $\tau > 0$ we define the extended cone
\begin{align*}
	&C^\tau_{\bar u} = \Big\{v \in L^2(Q) : |J'(\bar u)v| \le \tau\|v\|_{L^2(Q)} \text{ and }\\
	&\Big\{\begin{array}{cl}|j'(\bar u(t);v(t))| \le \tau\|v\|_{L^2(Q)}&\text{if } t \in I^+_\gamma,\\ j'(\bar u(t);v(t)) \le \tau\|v\|_{L^2(Q)}&\text{if } t \in I_\gamma \setminus I^+_\gamma,\end{array}\ \ \Big\}.
\end{align*}
\begin{theorem}
	Let $\bar u \in \uad$ satisfy the first order optimality conditions \eqref{E3.2}--\eqref{E3.5} and the second order condition $J''(\bar u)v^2 > 0$ $\forall v \in C_{\bar u} \setminus \{0\}$. Then, for every $r \in\big(\frac{4}{4-n},\infty]$ there exist strictly positive numbers $\varepsilon, \tau, \nu$ such that
	\begin{equation}
		J''(u)v^2 \ge \nu\|v\|^2_{L^2(Q)}\quad \forall v \in C^\tau_{\bar u}\ \text{ and } \ \forall u \in B_\varepsilon(\bar u),
		\label{E5.8}
	\end{equation}
	where $B_\varepsilon(\bar u)$ denotes the $L^r(0,T;L^2(\Omega))$ closed ball.
	\label{T5.3}
\end{theorem}

\begin{proof}
	First we prove the existence of $\tau > 0$ and $\nu > 0$ such that
	\begin{equation}
		J''(\bar u)v^2 \ge 2\nu\|v\|^2_{L^2(Q)} \quad \forall v \in C^\tau_{\bar u}.
		\label{E5.9}
	\end{equation}
	We proceed by contradiction. If \eqref{E5.9} fails for all strictly positive numbers $\tau, \nu$, then for every integer $k \ge 1$ there exists a function $v_k \in C^{\frac{1}{k}}_{\bar u}$ such that $J''(\bar u)v_k^2 < \frac{1}{k}\|v_k\|^2_{L^2(Q)}$. Dividing $v_k$ by its $L^2(Q)$ norm and taking a subsequence we get
	\begin{align}
		&\|v_k\|_{L^2(Q)} = 1,\quad v_k \rightharpoonup v \text{ in } L^2(Q),\quad J''(\bar u)v_k^2 < \frac{1}{k},\label{E5.10}\\
		&|J'(\bar u)v_k| \le \frac{1}{k},\quad \left\{\begin{array}{cl}|j'(\bar u(t);v_k(t))| \le \frac{1}{k}&\text{if } t \in I^+_\gamma,\\ j'(\bar u(t);v_k(t)) \le \frac{1}{k}&\text{if } t \in I_\gamma \setminus I^+_\gamma.\end{array}\right.\label{E5.11}
	\end{align}
	We prove that $v \in C_{\bar u}$. First, from \eqref{E5.10} and \eqref{E5.11} we get
	\[
	|J'(\bar u)v| \le \liminf_{k \to \infty}|J'(\bar u)v_k| \le 0.
	\]
	Thus, we have $J'(\bar u)v = 0$. Let us set
	\[
	I = \{t \in I_\gamma : j'(\bar u(t);v(t)) > 0\}.
	\]
	Then, we obtain with \eqref{E5.10} and \eqref{E5.11}
	\[
	\int_Ij'(\bar u(t);v(t))\dt \le \liminf_{k \to \infty}\int_Ij'(\bar u(t);v_k(t))\dt \le 0.
	\]
	This is not possible unless $|I| = 0$. Hence, we have that $j'(\bar u(t);v(t)) \le 0$ for almost all $t \in I_\gamma$. Now, from the identity $J'(\bar u)v = 0$, \eqref{E5.1}, and \eqref{E3.8} it follows
	\begin{align*}
		&0 = \int_Q(\bar\varphi + \kappa\bar u)v\dx\dt = -\int_Q\bar\mu v\dx \dt\\
		&= -\int_{I^+_\gamma}\left[\int_{\Omega^+_{\bar u(t)}}\|\bar\mu(t)\|_{L^\infty(\Omega)}v\dx - \int_{\Omega^-_{\bar u(t)}}\|\bar\mu(t)\|_{L^\infty(\Omega)}v\dx + \int_{\Omega^0_{\bar u(t)}}\mu v\dx\right]\dt.
	\end{align*}
	This implies
	\begin{equation}
		\int_{I^+_\gamma}\left[\int_{\Omega^+_{\bar u(t)}}\|\bar\mu(t)\|_{L^\infty(\Omega)}v\dx - \int_{\Omega^-_{\bar u(t)}}\|\bar\mu(t)\|_{L^\infty(\Omega)}v\dx\right]\dt = - \int_{I^+_\gamma}\int_{\Omega^0_{\bar u(t)}}\mu v \dx \dt.
		\label{E5.12}
	\end{equation}
	Now we have
	\begin{align*}
		&\int_{I^+_\gamma}\|\bar\mu(t)\|_{L^\infty(\Omega)} j'(\bar u(t);v(t))\dt\\
		&=\int_{I^+_\gamma}\|\bar\mu(t)\|_{L^\infty(\Omega)}\left[\int_{\Omega^+_{\bar u(t)}}v\dx - \int_{\Omega^-_{\bar u(t)}}v\dx + \int_{\Omega^0_{\bar u(t)}}|v| \dx\right]\dt.
	\end{align*}
	From this identity and \eqref{E5.12} we infer
	\begin{align*}
		&\int_{I^+_\gamma}\|\bar\mu(t)\|_{L^\infty(\Omega)} j'(\bar u(t);v(t))\dt = \int_{I^+_\gamma}\int_{\Omega^0_{\bar u(t)}}\big[\|\bar\mu(t)\|_{L^\infty(\Omega)}|v| - \bar\mu(t)v\big]\dx\dt \ge 0.
	\end{align*}
	This inequality along with $j'(\bar u(t);v(t)) \le 0$ for $t \in I_\gamma$ implies that $j'(\bar u(t);v(t)) = 0$ for almost all $t \in I^+_\gamma$. We have proved that $v \in C_{\bar u}$. From \eqref{E5.10} we infer
	\[
	J''(\bar u)v^2 \le \liminf_{k \to \infty}J''(\bar u)v_k^2 \le 0.
	\]
	Since $\bar u$ satisfies the second order condition, the above inequality is only possible if $v = 0$. Therefore, we have that $v_k \rightharpoonup 0$ in $L^2(Q)$.  Using \eqref{E2.21} and the fact that $z_{v_k}\to 0$ strongly in $L^2(Q)$ this yields
	\[
	\kappa = \liminf_{k \to \infty} \kappa\|v_k\|^2_{L^2(Q)} = \liminf_{k \to \infty}J''(\bar u)v_k^2 \le 0,
	\]
	which is a contradiction. Therefore, \eqref{E5.9} holds.
	
	Let us conclude the proof showing that \eqref{E5.9} implies \eqref{E5.8}. Given $\rho > 0$ arbitrarily small, from Theorem \ref{T2.2} we deduce the existence of $\varepsilon > 0$ such that
	\[
	\|y_u - \bar y\|_{L^\infty(Q)} = \|G(u) - G(\bar u)\|_{L^\infty(Q)} < \rho\quad \forall u \in B_\varepsilon(\bar u).
	\]
	Using this estimate, we get from \eqref{E2.17} and \eqref{E2.22}, and taking a smaller $\varepsilon$ if necessary
	\[
	\|\varphi_u - \bar\varphi\|_{L^\infty(Q)} + \|z_{u,v} - z_v\|_{L^2(Q)}  < \rho \quad \forall u \in B_\varepsilon(\bar u)\ \text{ and }\ \forall v \in L^2(Q),
	\]
	where $z_{u,v} = G'(u)v$, $z_v = G'(\bar u)v$, and $\varphi_u$ and $\bar\varphi$ are the adjoint states corresponding to $u$ and $\bar u$, respectively. Therefore, selecting $\rho$ small enough we obtain with \eqref{E2.21} for some $\varepsilon > 0$
	\[
	|[J''(u) - J''(\bar u)]v^2| \le \nu\|v\|^2_{L^2(Q)} \quad \forall u \in B_\varepsilon(\bar u) \ \text{ and }\ \forall v \in L^2(Q).
	\]
	Combining this with \eqref{E5.9} we infer \eqref{E5.8}.
	\qed\end{proof}

\section{Stability of the optimal controls with respect to $\gamma$}
\label{S6}

The aim of this section is to prove some stability of the local or global solutions of \Pb with respect to $\gamma$. For every $\gamma > 0$ we consider the control problems
\[
\Pbg \quad  \inf_{u \in \uadg \cap L^\infty(Q)} J(u),
\]
where
\[
\uadg = \{u \in L^\infty(0,T;L^1(\Omega)) : \|u(t)\|_{L^1(\Omega)} \le \gamma \text{ for a.a. } t \in (0,T)\}.
\]
First, we prove some continuity of the solutions of \Pbg with respect to $\gamma$.

\begin{theorem}
	Let $\{\gamma_k\}_{k = 1}^\infty \subset (0,\infty)$ be a sequence converging to some $\gamma > 0$. For every $k$ let $u_{\gamma_k}$ be a global minimizer of the problem \Pbgk. Then, the sequence $\{u_{\gamma_k}\}_{k = 1}^\infty$ is bounded in $L^\infty(Q)$. Moreover, if $u_\gamma$ is a weak$^*$ limit in $L^\infty(Q)$ of a subsequence of $\{u_{\gamma_k}\}_{k = 1}^\infty$, then $u_\gamma$ is a global minimizer of \Pbg and the convergence is strong in $L^p(Q)$ for every $p < \infty$. Reciprocally, for every strict local minimizer $u_\gamma$ of \Pbg in the $L^r(0,T;L^2(\Omega))$ sense with $\frac{4}{4-n} < r < \infty$, there exists a sequence $\{u_{\gamma_k}\}_{k = 1}^\infty$ such that $u_{\gamma_k}$ is a $L^r(0,T;L^2(\Omega))$ local minimizer of \Pbgk and $u_{\gamma_k} \to u_\gamma$ strongly in $L^p(Q)$ for every $p < \infty$.
	\label{T6.1}
\end{theorem}

\begin{proof}
	The boundedness of $\{u_{\gamma_k}\}_{k = 1}^\infty$ in $L^\infty(Q)$ follows from Theorem \ref{T3.2}. Therefore, we can take subsequences converging weakly$^*$ in $L^\infty(Q)$. Let us take one of these subsequences, that we denote in the same form, such that $u_{\gamma_k} \stackrel{*}{\rightharpoonup} \hat u$ in $L^\infty(Q)$. Let $u_\gamma$ be a solution of \Pbg. For every $k$ we define
	\begin{equation}
		u_k = \left\{\begin{array}{cl} u_\gamma&\text{if } \gamma \le \gamma_k,\vspace{2mm}\\\frac{\gamma_k}{\gamma}u_\gamma&\text{if } \gamma > \gamma_k,\end{array}\right.\quad \text{ and }\quad \hat u_k = \left\{\begin{array}{cl} u_{\gamma_k}&\text{if } \gamma_k \le \gamma,\vspace{2mm}\\\frac{\gamma}{\gamma_k}u_{\gamma_k}&\text{if } \gamma_k > \gamma.\end{array}\right.
		\label{E6.1}
	\end{equation}
	Then, it is immediate that $u_k \to u_\gamma$ and $\hat u_k \stackrel{*}{\rightharpoonup} \hat u$ in $L^\infty(Q)$, $\{\hat u_k\}_{k = 1}^\infty \subset \uadg$ and $u_k \in \uadgk \cap \uadg$ for every $k$. Since $\uadg \cap L^2(Q)$ is a closed and convex subset of $L^2(Q)$ and $\hat u_k \rightharpoonup \hat u$ in $L^2(Q)$, we deduce that $\hat u \in \uadg$. With the compactness of the embedding $W(0,T) \subset L^2(Q)$ we can easily prove that $y_{\hat u_k} \to y_{\hat u}$ in $L^2(Q)$. Using these convergences and the optimality of $u_{\gamma_k}$ and $u_\gamma$ we get
	\[
	J(u_\gamma) \le J(\hat u) \le \liminf_{k \to \infty}J(u_{\gamma_k}) \le \limsup_{k \to \infty}J(u_{\gamma_k}) \le \limsup_{k \to \infty} J(u_k) = J(u_\gamma).
	\]
	This implies that $J(u_\gamma) = J(\hat u) = \lim_{k \to \infty}J(u_{\gamma_k})$. This identity proves that $\hat u$ is a solution of \Pbg. Moreover, the convergence $y_{u_{\gamma_k}} \to y_{u_\gamma}$ in $L^2(Q)$ leads to $\lim_{k \to \infty}\|u_{\gamma_k}\|_{L^2(Q)} = \|\hat u\|_{L^2(Q)}$. From this fact and the weak convergence $u_{\gamma_k} \rightharpoonup u_\gamma$ in $L^2(Q)$, we obtain that $u_{\gamma_k} \to \hat u$ in $L^2(Q)$. This along with the boundedness of $\{u_{\gamma_k}\}_{k = 1}^\infty$ in $L^\infty(Q)$ implies the strong convergence in $L^p(Q)$ for every $p < \infty$.
	
	Let us prove the second part of the theorem. Let $u_\gamma$ be an $L^r(0,T;L^2(\Omega))$ strict local minimizer to \Pbg. This means that there exists $\varepsilon > 0$ such that
	\[
	J(u_\gamma) < J(u) \quad \forall u \in \uadg \cap B_\varepsilon(u_\gamma) \text{ with } u \neq u_\gamma,
	\]
	where $B_\varepsilon(u_\gamma)$ is the closed ball in $L^r(0,T;L^2(\Omega))$ of radius $\varepsilon$ and center $u_\gamma$. Now, we consider the problems
	\[
	(PB_\gamma) \ \ \min_{u \in \uadg \cap B_\varepsilon(u_\gamma)}J(u) \qquad \text{ and }\qquad (PB_{\gamma_k})\ \ \min_{u \in \uadgk \cap B_\varepsilon(u_\gamma)}J(u)
	\]
	It is immediate that $u_\gamma$ is the unique solution of $(PB_\gamma)$. Observe that the controls $u_k$ defined in \eqref{E6.1} are elements of $\uadgk \cap B_\varepsilon(u_\gamma)$ for all $k$ large enough. Hence,  $\uadgk \cap B_\varepsilon(u_\gamma)$ is non-empty, closed, convex, and bounded in $L^r(0,T;L^2(\Omega))$. Therefore, problem $(PB_{\gamma_k})$ has at least one solution $u_{\gamma_k}$. Let us prove that $u_{\gamma_k} \to u_\gamma$ in $L^p(Q)$ for every $p < \infty$. Denote $y_{\gamma_k}$ and $\varphi_{\gamma_k}$ the state and adjoint state associated with $u_{\gamma_k}$. Since $\{u_{\gamma_k}\}_{k = 1}^\infty$ is bounded in $L^r(0,T;L^2(\Omega))$ we infer from Theorem \ref{T2.1} the boundedness of $\{y_{\gamma_k}\}_{k = 1}^\infty$ in $L^\infty(Q)$. Hence, from the adjoint state equation and the classical estimates for linear equations we deduce that $\{\varphi_{\gamma_k}\}_{k = 1}^\infty$ is also bounded in $L^\infty(Q)$. Due to the optimality of $u_{\gamma_k}$ for $(PB_{\gamma_k})$ we obtain
	\[
	\int_Q(\varphi_{\gamma_k} + \kappa u_{\gamma_k})(u - u_{\gamma_k})\dx \dt = J'(u_{\gamma_k})(u - u_{\gamma_k}) \ge 0 \ \ \forall u \in \uadgk \cap B_\varepsilon(u_\gamma).
	\]
	Setting $S = \uadgk \cap B_\varepsilon(u_\gamma)$ we get from the above inequalities
	\[
	u_{\gamma_k} = \proj_S\big(-\frac{1}{\kappa}\varphi_{\gamma_k}\big),
	\]
	where $\proj_S$ denotes the $L^2(Q)$ projection on $S$. Let us prove that
	\begin{equation}
		\|u_{\gamma_k}\|_{L^\infty(Q)} \le 2\big(\frac{1}{\kappa}\|\varphi_{\gamma_k}\|_{L^\infty(Q)} + \|u_\gamma\|_{L^\infty(Q)}\big).
		\label{E6.2}
	\end{equation}
	For this purpose we define
	\[
	Q_0 = \big\{(x,t) \in Q : |u_{\gamma_k}(x,t)| > 2\big(\frac{1}{\kappa}|\varphi_{\gamma_k}(x,t)| + |u_\gamma(x,t)|\big)\big\}.
	\]
	Put
	\[
	u(x,t) = \left\{\begin{array}{cl}-\frac{1}{\kappa}\varphi_{\gamma_k}(x,t) + u_\gamma(x,t) &\text{if } (x,t) \in Q_0,\\u_{\gamma_k}(x,t)&\text{otherwise.}\end{array}\right.
	\]
	Then, it is obvious that
	\begin{align*}
		&\|u(t)\|_{L^1(\Omega)} \le \|u_{\gamma_k}(t)\|_{L^1(\Omega)} \le \gamma_k,\\
		&\|u - u_\gamma\|_{L^r(0,T;L^2(\Omega))} \le \|u_{\gamma_k} - u_\gamma\|_{L^r(0,T;L^2(\Omega))} \le \varepsilon,\\
		&\big\|u + \frac{1}{\kappa}\varphi_{\gamma_k}\big\|_{L^2(Q)} < \big\|u_{\gamma_k} + \frac{1}{\kappa}\varphi_{\gamma_k}\big\|_{L^2(Q)}\ \text{ if } |Q_0| \neq 0,
	\end{align*}
	The first two inequalities show that $u \in S$ and, consequently, the third one contradicts the fact that $u_{\gamma_k}$ is the $L^2(Q)$ projection of $-\frac{1}{\kappa}\varphi_{\gamma_k}$ unless $|Q_0| = 0$. Now, the boundedness of $\{\varphi_{\gamma_k}\}_{k = 1}^\infty$ in $L^\infty(Q)$ and \eqref{E6.2} imply the boundedness of $\{u_{\gamma_k}\}_{k = 1}^\infty$. Therefore, there exists a subsequence, denoted in the same way, such that $u_{\gamma_k} \stackrel{*}{\rightharpoonup} \hat u$ in $L^\infty(Q)$. Using the functions $\{\hat u_k\}_{k = 1}^\infty$ defined in \eqref{E6.1} and arguing as above, we deduce that $\hat u \in \uadg$. Moreover, is is also immediate that $\hat u \in B_\varepsilon(u_\gamma)$. Let us consider the functions $\{u_k\}_{k = 1}^\infty$ defined in \eqref{E6.1}. Since
	\[
	\|u_k - u_\gamma\|_{L^\infty(Q)} = \left\{\begin{array}{cl}0 &\text{if } \gamma \le \gamma_k,\\ \frac{\gamma - \gamma_k}{\gamma}\|u_\gamma\|_{L^\infty(Q)}&\text{otherwise,}\end{array}\right.
	\]
	we have that $u_k \to u_\gamma$ in $L^\infty(Q)$ as $k \to \infty$ and $u_k \in \uadgk \cap B_\varepsilon(u_\gamma)$ for every $k$ large enough. Then, using the optimality of $u_\gamma$ and $u_{\gamma_k}$, and the fact that $u_k$ and $\hat u$ are feasible controls for $(PB_{\gamma_k})$ and $(PB_\gamma)$, respectively, we infer
	\[
	J(u_\gamma) \le J(\hat u) \le \liminf_{k \to \infty}J(u_{\gamma_k}) \le \limsup_{k \to \infty}J(u_{\gamma_k}) \le \limsup_{k \to \infty} J(u_k) = J(u_\gamma).
	\]
	This implies that $J(u_\gamma) = J(\hat u)$ and, hence, $\hat u$ is also a solution of $(PB_\gamma)$. Due to the uniqueness of solution of $(PB_\gamma)$ we conclude that $u_\gamma = \hat u$. The strong convergence $u_{\gamma_k} \to u_\gamma$ in $L^p(Q)$ follows as above. We have proved that every subsequence converge to $u_\gamma$, then the whole sequence does. In particular, the convergence $u_{\gamma_k} \to u_\gamma$ in $L^r(0,T;L^2(\Omega))$ implies that $u_{\gamma_k}$ is in the interior of the ball $B_\varepsilon(u_\gamma)$ for all $k$ sufficiently large. Hence, $u_{\gamma_k}$ is an $L^r(0,T;L^2(\Omega))$ local minimizer of $(PB_{\gamma_k})$.
	\qed\end{proof}

\begin{remark}
	Given an $L^r(0,T;L^2(\Omega))$ strict local minimizer of \Pbg, from the above theorem we deduce the existence of a family $\{u_{\gamma'}\}_{\gamma' > 0}$ of $L^r(0,T;L^2(\Omega))$ local minimizers of problems \Pbgp such that $u_{\gamma'} \to u_\gamma$ in $L^p(Q)$ as $\gamma' \to \gamma$ for every $p < \infty$. Looking at the definition of the elements $u_{\gamma_k}$ in the previous proof  we have that
	\begin{equation}
		J(u_{\gamma'}) \le J(u) \ \ \forall u \in \uadgp \cap B_\varepsilon(u_\gamma)\quad \text{and}\quad J(u_\gamma) \le J(u) \ \ \forall u \in \uadg \cap B_\varepsilon(u_\gamma).
		\label{E6.3}
	\end{equation}
	\label{R6.1}
\end{remark}

\begin{theorem}
	Let $\{u_{\gamma'}\}_{\gamma'}$ be a family of local minimizers of problems \Pbgp such that $u_{\gamma'} \to u_\gamma$ in $L^r(0,T;L^2(\Omega))$ as $\gamma' \to \gamma$ with $u_\gamma$ a local minimizer of \Pbg satisfying \eqref{E5.3}. We also assume that \eqref{E6.3} holds. Then, there exists a constant $L$ such that
	\begin{equation}
		\|u_{\gamma'} - u_\gamma\|_{L^2(Q)} \le L|\gamma' - \gamma|^{\frac{1}{2}}.
		\label{E6.4}
	\end{equation}
	\label{T6.2}
\end{theorem}

\begin{proof}
	The first part of the theorem follows from Remark \ref{R6.1}. We only have to prove \eqref{E6.4}. For every $\gamma'$ we define
	\begin{equation}
		\hat u_{\gamma'} = \left\{\begin{array}{cl} u_\gamma&\text{if } \gamma < \gamma',\vspace{2mm}\\\frac{\gamma'}{\gamma}u_\gamma&\text{if } \gamma > \gamma',\end{array}\right.\quad \text{ and }\quad \hat v_{\gamma'} = \left\{\begin{array}{cl} u_{\gamma'}&\text{if } \gamma' < \gamma,\vspace{2mm}\\\frac{\gamma}{\gamma'}u_{\gamma'}&\text{if } \gamma' > \gamma.\end{array}\right.
		\label{E6.5}
	\end{equation}
	Then we have
	\begin{equation}
		\hat u_{\gamma'}, \hat v_{\gamma'} \in \uadg \cap \uadgp,\ \ \hat u_{\gamma'} \to u_\gamma \text{ in } L^\infty(Q) \text{ and } \hat v_{\gamma'} \to u_\gamma \text{ in } L^r(0,T;L^2(\Omega)).
		\label{E6.6}
	\end{equation}
	From here we infer that $v_{\gamma'} \in \uadg \cap B_\varepsilon(u_\gamma)$ for $\gamma'$ close enough to $\gamma$ with $B_\varepsilon(u_\gamma)$ defined in \eqref{E5.3}. Therefore, we get
	\begin{equation}
		\frac{\delta}{2}\|\hat v_{\gamma'} - u_\gamma\|^2_{L^2(Q)} \le J(\hat v_{\gamma'}) - J(u_\gamma).
		\label{E6.7}
	\end{equation}
	In the case $\gamma' < \gamma$, using \eqref{E6.7}, the optimality of $u_{\gamma'}$, and the definition of $\hat v_{\gamma'}$ we obtain with the mean value theorem
	\begin{align*}
		&\|u_{\gamma'} - u_\gamma\|^2_{L^2(Q)} \le \frac{2}{\delta}\big[(J(u_{\gamma'}) - J(\hat u_{\gamma'})) + (J(\hat u_{\gamma'}) - J(u_\gamma))\big]\\
		&\le \frac{2}{\delta}(J(\hat u_{\gamma'}) - J(u_\gamma)) \le C_1\|\hat u_{\gamma'} - u_\gamma\|_{L^\infty(Q)} = \frac{C_1}{\gamma}\|u_\gamma\|_{L^\infty(Q)}|\gamma' - \gamma|.
	\end{align*}
	In the case $\gamma' > \gamma$ we proceed as follows
	\begin{align*}
		&\|\hat v_{\gamma'} - u_\gamma\|^2_{L^2(Q)} \le \frac{2}{\delta}\big[(J(\hat v_{\gamma'}) - J(u_{\gamma'})) + (J(u_{\gamma'}) - J(u_\gamma))\big]\\
		&\le \frac{2}{\delta}\big(J(\hat v_{\gamma'}) - J(u_{\gamma'})\big) \le C_2\|\hat v_{\gamma'} - u_{\gamma'}\|_{L^r(0,T;L^2(\Omega))}\\
		& = \frac{C_2}{\gamma'}\|u_{\gamma'}\|_{L^r(0,T;L^2(\Omega))}|\gamma' - \gamma| \le C_3|\gamma' - \gamma|.
	\end{align*}
	From here we get
	\begin{align*}
		&\|u_{\gamma'} - u_\gamma\|_{L^2(Q)} \le \|u_{\gamma'} - \hat v_{\gamma'}\|_{L^2(Q)} + \|\hat v_{\gamma'} - u_{\gamma}\|_{L^2(Q)}\\
		&\le \frac{|\gamma' - \gamma|}{\gamma'}\|u_{\gamma'}\|_{L^2(Q)} + \sqrt{C_3}|\gamma' - \gamma|^{\frac{1}{2}} \le C_4|\gamma' - \gamma|^{\frac{1}{2}},
	\end{align*}
	which concludes the proof.
	\qed\end{proof}

Theorems \ref{T5.2} and \ref{T6.2} imply H\"older stability with respect to $\gamma$ of the optimal controls if the sufficient second order condition $J''(u_\gamma)v^2 > 0$ $\forall v \in C_{\bar u} \setminus \{0\}$ holds. Now, we are interested in proving Lipschitz stability. To this end we need to make a stronger assumption, namely
\begin{equation}
	J''(u_{\gammab})v^2 > 0\quad \forall v \in L^2(Q) \setminus \{0\}, \quad y_0 \in C_0(\Omega), \quad \text{and}\quad \rd > \frac{4}{4-n},
	\label{E6.8}
\end{equation}
where $\gammab > 0$ is fixed and $C_0(\Omega)$ denotes the space of continuous real valued  functions on $\bar\Omega$ vanishing on $\Gamma$. From the first assumption in \eqref{E6.8} we deduce the existence of strictly positive numbers $\rho$ and $\nu$ such that
\begin{equation}
	J''(u)v^2 \ge \nu\|v\|^2_{L^2(Q)}\quad \forall v \in L^2(Q)\ \text{ and } \ \forall u \in B_\rho(u_{\gammab}),
	\label{E6.9}
\end{equation}
where $B_\rho(u_{\gammab})$ denotes the $L^\rd(0,T;L^2(\Omega))$ closed ball. Indeed, if \eqref{E6.9} does not hold, then we can take sequences $\{u_k\}_{k = 1}^\infty \subset L^\rd(0,T;L^2(\Omega))$ and $\{v_k\}_{k = 1}^\infty \subset L^2(Q)$ satisfying
\[
\lim_{k \to \infty}\|u_k - u_{\gammab}\|_{L^\rd(0,T;L^2(\Omega))} = 0, \|v_k\|_{L^2(Q)} = 1, v_k \rightharpoonup v \text{ in } L^2(Q), J''(u_k)v_k^2 \le \frac{1}{k}.
\]
It is easy to pass to the limit and to deduce
\[
J''(u_{\gammab})v^2 \le \liminf_{k \to \infty}J''(u_k)v_k^2 \le 0.
\]
This inequality and \eqref{E6.8} yield $v = 0$. But, arguing as in the proof of Theorem \ref{T5.3} we infer
\[
\kappa = \liminf_{k \to \infty}\kappa\|v_k\|^2_{L^2(Q)} = \liminf_{k \to \infty}J''(u_k)v_k^2 \le 0,
\]
which contradicts our assumption $\kappa > 0$.

We finish this section by proving the next theorem.

\begin{theorem}
	Let $u_{\gammab}$ be a local minimizer of (P$_{\gammab}$). We assume that \eqref{E6.8} holds and that $\rho$ satisfies \eqref{E6.9}. Then, there exists $\bar \varepsilon\in(0,\bar\gamma)$ such that $\Pbg$ has a unique local minimizer $u_\gamma$ in the interior of the $L^\rd(0,T;L^2(\Omega))$ ball $B_\rho(u_{\gammab})$ for every $\gamma \in (\gammab - \bar\varepsilon,\gammab + \bar\varepsilon)$. Moreover, there exists a constant $L$ such that
	\begin{equation}
		\|u_\gamma - u_{\gammab}\|_{L^\rd(0,T;L^2(\Omega))} \le L|\gamma - \gammab|\quad \forall \gamma \in (\gammab - \bar\varepsilon,\gammab + \bar\varepsilon).
		\label{E6.10}
	\end{equation}
	\label{T6.3}
\end{theorem}

\begin{proof}
	Let us take $\rho > 0$ such that \eqref{E6.9} holds. Then, $J$ has at most one local (and global) minimizer $u_\gamma$ in the closed set $B_\rho(u_{\gammab}) \cap \uad$. This is a consequence of the strict convexity of $J$ in the ball $B_\rho(u_{\gammab})$; see \eqref{E6.9}. We will prove that this local minimizer belongs to the interior of the $L^\rd(0,T;L^2(\Omega))$ ball $B_\rho(u_{\gammab})$ if $\gamma$ is close enough to $\gammab$, and consequently it is a local minimizer of \Pbg. In order to prove this, as well as \eqref{E6.10}, we reformulate the control problem \Pbg as follows
	\[
	\Qbg \quad  \inf_{u \in \kad} J_\gamma(u):= \frac{1}{2}\int_Q (y_{\gamma,u}(x,t) - y_d(x))^2\dx\dt + \frac{\kappa\gamma^2}{2}\int_Q u(x,t)^2\dx\dt,
	\]
	where
	\[
	\kad = \{u \in L^\rd(0,T;L^2(\Omega)) : \|u(t)\|_{L^1(\Omega)} \le 1 \text{ for a.a. } t \in (0,T)\}
	\]
	and $y_{\gamma,u}$ is the solution of the semilinear parabolic equation
	\begin{equation}
		\left\{\begin{array}{ll}\displaystyle\frac{\partial y}{\partial t} + Ay + a(x,t,y) =  \gamma u & \mbox{in } Q = \Omega \times (0,T),\\y = 0  \mbox{ on } \Sigma = \Gamma \times (0,T),& y(0) = y_0 \mbox{ in } \Omega.\end{array}\right.
		\label{E6.11}
	\end{equation}
	It is obvious that the problems \Pbg and \Qbg are equivalent for every $\gamma$. This equivalence is understood in the sense that $u$ is a local (global) minimizer of \Qbg if and only if $u_\gamma = \gamma u$ is a local (global) minimizer of \Pbg, and $J(u_\gamma) = J_\gamma(u)$; recall Remark \ref{R4.2}.
	
	Take $\varepsilon \in (0,\gammab)$ and $\bar\rho \in (0,\rho]$ such that $(\gammab + \varepsilon)\bar\rho + \varepsilon\|\bar u\|_{L^\rd(0,T;L^2(\Omega))} < \rho$. Then, we have with the notation $u_\gammab = \gammab\bar u$ and $u_\gamma = \gamma u$
	\begin{align*}
		&\|u_\gamma - u_\gammab\|_{L^\rd(0,T;L^2(\Omega)} \le \gamma\|u - \bar u\|_{L^\rd(0,T;L^2(\Omega))} + |\gamma - \gammab|\|\bar u\|_{L^\rd(0,T;L^2(\Omega))}\\
		& \le (\gammab + \varepsilon)\bar\rho + \varepsilon\|u_\gammab\|_{L^\rd(0,T;L^2(\Omega))} < \rho \ \ \forall u \in B_{\bar\rho}(\bar u) \text{ and } \forall \gamma \in (\gammab - \varepsilon,\gammab + \varepsilon).
	\end{align*}
	Due to \eqref{E6.9} and the fact that $J_\gamma''(u)v^2 = \gamma^2 J''(u_\gamma)v^2$, we deduce that
	\[
	J''_\gamma(u)v^2 \ge \gamma^2\nu\|v\|^2_{L^\rd(0,T;L^2(\Omega))} \ge (\gammab - \varepsilon)^2\nu\|v\|^2_{L^\rd(0,T;L^2(\Omega))}\quad \forall u \in B_{\bar\rho}(\bar u).
	\]
	Therefore, $J_\gamma$ is strictly convex on the ball $B_{\bar\rho}(\bar u)$. Hence, a control $u$ is a local solution of \Qbg in the interior of $B_{\bar\rho}(\bar u)$ if and only if $u$ satisfies the optimality system
	\begin{align}
		&\left\{\begin{array}{l}\displaystyle\frac{\partial y}{\partial t} + Ay + a(x,t,y) =  \gamma u \ \mbox{ in } Q,\\ y = 0  \mbox{ on } \Sigma,\ \ y(0) = y_0 \mbox{ in } \Omega,\end{array}\right.
		\label{E6.12}\\
		&\left\{\begin{array}{l}\displaystyle-\frac{\partial\varphi}{\partial t} + A^*\varphi + \frac{\partial a}{\partial y}(x,t,y)\varphi =  y - y_d\ \mbox{ in } Q,\\ \varphi = 0  \mbox{ on } \Sigma,\ \ \varphi(T) = 0 \mbox{ in } \Omega,\end{array}\right.
		\label{E6.13}\\
		&\int_Q\mu(v - u)\dx\dt \le 0\quad \forall v \in \kad, \label{E6.14}\\
		&\gamma\varphi + \kappa\gamma^2 u + \mu = 0. \label{E6.15}
	\end{align}
	Denote by $\bar y$ and $\bar\varphi$ the state and adjoint state associated to $\bar u$. Our goal is to apply \cite[Theorem 2.4]{Dontchev1995} to the previous optimality system. To this end we define the spaces:
	\begin{align*}
		& \mathcal{Y} = \{y \in W(0,T) \cap C(\bar Q) : \frac{\partial y}{\partial t} + Ay \in L^\rd(0,T;L^2(\Omega))\},\\
		& \Phi = \{\varphi \in H^1(Q) \cap C(\bar Q) : -\frac{\partial\varphi}{\partial t} + A^*\varphi \in L^\rd(0,T;L^2(\Omega)) \text{ and } \varphi(T) = 0\},\\
		& X = \mathcal{Y} \times \Phi \times L^\rd(0,T;L^2(\Omega)),\ \ Y = \mathbb{R},\ \ Z = C_0(\Omega) \times L^\rd(0,T;L^2(\Omega))^3.
	\end{align*}
	On $\mathcal{Y}$ and $\Phi$ we consider the graph norms
	\begin{align*}
		&\|y\|_{\mathcal{Y}} = \|y\|_{W(0,T)} + \|y\|_{C(\bar Q)} + \Big\|\frac{\partial y}{\partial t} + Ay\Big\|_{L^\rd(0,T;L^2(\Omega))},\\
		&\|\varphi\|_\Phi = \|\varphi\|_{H^1(Q)} + \|\varphi\|_{C(\bar Q)} + \Big\|-\frac{\partial \varphi}{\partial t} + A^*\varphi\Big\|_{L^\rd(0,T;L^2(\Omega))}.
	\end{align*}
	Thus, $X$ is a Banach space. Moreover, we introduce the mapping $f:X \times Y \longrightarrow Z$ and the multivalued function $F:X \longrightarrow Z$
	\[
	f((y,\varphi,u),\gamma) = \left(\begin{array}{c}y(0) - y_0\\\displaystyle\frac{\partial y}{\partial t} + Ay + a(\cdot,\cdot,y) - \gamma u\\\displaystyle-\frac{\partial\varphi}{\partial t} + A^*\varphi + \frac{\partial a}{\partial y}(\cdot,\cdot,y)\varphi - y + y_d\\\gamma\varphi + \gamma^2\kappa u\end{array}\right), \quad F(y,\varphi,u) = \left(\begin{array}{c} 0\\0\\0\\F_0(u)\end{array}\right),
	\]
	where the multivalued function $F_0:L^\rd(0,T;L^2(\Omega)) \longrightarrow L^\rd(0,T;L^2(\Omega))$ is defined by
	\[
	F_0(u) = \left\{\begin{array}{cl} \emptyset &\text{if } u \not\in \kad,\vspace{2mm}\\\displaystyle\Big\{\mu \in L^\rd(0,T;L^2(\Omega)) : \int_Q\mu(v - u)\dx\dt \le 0\ \forall v \in \kad\Big\},&\text{otherwise.}\end{array}\right.
	\]
	
	Due to the regularity $y_0 \in C_0(\Omega)$, see assumption \eqref{E6.8}, we deduce from \eqref{E6.12} that $\bar y \in \mathcal{Y}$. Therefore, we have that $(\bar y,\bar\varphi,\bar u) \in X$. Moreover, $(\bar y,\bar\varphi,\bar u)$ satisfies the optimality system \eqref{E6.12}--\eqref{E6.15}, which implies that $0 \in f((\bar y,\bar\varphi,\bar u),\gammab) + F(\bar y,\bar\varphi,\bar u)$. Using our assumptions on $a$ and the continuous embedding $\mathcal{Y} \subset C(\bar Q)$ we deduce that the function $f$ is of class $C^1$. Then, the function $g:X \longrightarrow Z$, defined by
	\[
	g(y,\varphi,u) = f((\bar y,\bar\varphi,\bar u),\gammab) + D_{(y,\varphi,u)}f((\bar y,\bar\varphi,\bar u),\gammab)(y - \bar y,\varphi - \bar\varphi,u - \bar u),
	\]
	strongly approximates $f$ at $((\bar y,\bar\varphi,\bar u),\gammab)$, and $g(\bar y,\bar\varphi,\bar u) = f((\bar y,\bar\varphi,\bar u),\gammab)$; see \cite{Robinson91} for the definition of a strong approximation.
	
	We will apply \cite[Theorem 2.4]{Dontchev1995} to deduce the existence of $\bar\varepsilon \in (0,\varepsilon]$ and $\tilde\rho \in (0,\bar\rho]$ such that \eqref{E6.12}--\eqref{E6.15} has a unique solution $u$ in the interior of the ball $B_{\tilde\rho}(\bar u)$ for every $\gamma \in (\gammab - \bar\varepsilon,\gammab + \bar\varepsilon)$. Moreover, these solutions satisfy
	\begin{equation}
		\|u - \bar u\|_{L^\rd(0,T;L^2(\Omega))} \le \lambda|\gamma - \gammab|.
		\label{E6.16}
	\end{equation}
	for some $\lambda > 0$. For this purpose it is enough to prove that the equation
	\begin{equation}
		\beta \in g(y,\varphi,u) + F(y,\varphi,u)
		\label{E6.17}
	\end{equation}
	has a unique solution $(y_\beta,\varphi_\beta,u_\beta) \in X$ for every $\beta = (\beta_i)_{i = 1}^4 \in Z$ and the Lipschitz property
	\begin{equation}
		\|(y_{\hat\beta},\varphi_{\hat\beta},u_{\hat\beta}) - (y_\beta,\varphi_\beta,u_\beta)\|_X \le \lambda\|\hat\beta - \beta\|_Z
		\label{E6.18}
	\end{equation}
	holds for some $\lambda > 0$ and all $\hat\beta, \beta \in Z$. First, we prove the existence of a unique solution. To this end we consider the optimal control problem
	\[
	\text{\rm (P$_\beta$)}\quad \inf_{u \in \kad} \mathcal{J}_\beta(u),
	\]
	where
	\begin{align*}
		\mathcal{J}_\beta(u):=& \frac{1}{2}\int_Q \big[1 - \frac{\partial^2a}{\partial y^2}(x,t,\bar y)\bar\varphi\big]y^2\dx\dt + \frac{\kappa\gammab^2}{2}\int_Q u(x,t)^2\dx\dt\\
		& + \int_Q\beta_3 y\dx\dt + \int_Q(\gammab\bar\varphi + \gammab^2\kappa\bar u - \beta_4)u\dx\dt,
	\end{align*}
	and $y$ satisfies the equation
	\begin{equation}
		\left\{\begin{array}{l}\displaystyle\frac{\partial y}{\partial t} + A y + \frac{\partial a}{\partial y}(x,t,\bar y)y =  \gammab u + \beta_2\ \mbox{ in } Q,\\ y = 0  \mbox{ on } \Sigma,\ \ y(0) = \beta_1 \mbox{ in } \Omega.\end{array}\right.
		\label{E6.19}
	\end{equation}
	Let us consider the solution $\xi_\beta \in \mathcal{Y}$ of the equation
	\begin{equation}
		\left\{\begin{array}{l}\displaystyle\frac{\partial\xi}{\partial t} + A\xi + \frac{\partial a}{\partial y}(x,t,\bar y)\xi = \beta_2\ \mbox{ in } Q,\\ \xi = 0  \mbox{ on } \Sigma,\ \ \xi(0) = \beta_1 \mbox{ in } \Omega.\end{array}\right.
		\label{E6.20}
	\end{equation}
	According to \eqref{E2.17} we have that $y = \gammab G'(u_{\gammab})u + \xi_\beta = \gammab z_u + \xi_\beta$. Inserting this identity in the cost functional we get
	\begin{align*}
		\mathcal{J}_\beta(u)=& \frac{\gammab^2}{2}\left\{\int_Q \big[1 - \frac{\partial^2a}{\partial y^2}(x,t,\bar y)\bar\varphi\big]z_u^2\dx\dt + \kappa\int_Q u^2\dx\dt\right\}\\
		& + \gammab\int_Q\Big(\big[1 - \frac{\partial^2a}{\partial y^2}(x,t,\bar y)\bar\varphi\big]\xi_\beta + \beta_3\Big)z_u\dx\dt + \int_Q(\gammab\bar\varphi + \gammab^2\kappa u_{\gammab} - \beta_4)u\dx\dt\\
		& + \int_Q\Big(\frac{1}{2}\big[1 - \frac{\partial^2a}{\partial y^2}(x,t,\bar y)\bar\varphi\big]\xi_\beta^2 + \beta_3\xi_\beta\Big)\dx\dt.
	\end{align*}
	From \eqref{E2.21}, \eqref{E6.9}, and the continuity of the mapping $u \to z_u$ in $L^2(Q)$ we deduce the existence of two constants $C_1$ and $C_2$ such that
	\[
	\mathcal{J}_\beta(u) \ge \frac{\gammab^2}{2}\nu\|u\|^2_{L^2(Q)} + C_1\|u\|_{L^2(Q)} + C_2.
	\]
	Therefore, $\mathcal{J}_\beta$ is a coercive, continuous, and strictly convex quadratic functional on $L^2(Q)$. As a consequence, we infer the existence and uniqueness of a minimizer $\tilde u_\beta$ of $\mathcal{J}_\beta$ on the set
	\[
	\tilde{K}_{ad} = \{u \in L^2(Q) : \|u(t)\|_{L^1(\Omega)} \le 1 \text{ for a.a. } t \in (0,T)\}.
	\]
	Similarly as in Theorem \ref{T3.1}, we deduce the existence of elements $\tilde y_\beta \in W(0,T)$, $\tilde\varphi_\beta \in H^1(Q)$, and $\tilde\mu_\beta \in L^2(Q)$ satisfying
	\begin{align}
		&\left\{\begin{array}{l}\displaystyle\frac{\partial\tilde y_\beta}{\partial t} + A\tilde y_\beta + \frac{\partial a}{\partial y}(x,t,\bar y)\tilde y_\beta = \gammab\tilde u_\beta + \beta_2\ \mbox{ in } Q,\\ \tilde y_\beta = 0  \mbox{ on } \Sigma,\ \ \tilde y_\beta(0) = \beta_1 \mbox{ in } \Omega,\end{array}\right.
		\label{E6.21}\\
		&\left\{\begin{array}{l}\displaystyle-\frac{\partial\tilde\varphi_\beta}{\partial t} + A^*\tilde\varphi_\beta + \frac{\partial a}{\partial y}(x,t,\bar y)\tilde\varphi_\beta =  \big[1 - \frac{\partial^2a}{\partial y^2}(x,t,\bar y)\bar\varphi\big]\tilde y_\beta + \beta_3\ \mbox{ in } Q,\\ \tilde\varphi_\beta = 0  \mbox{ on } \Sigma,\ \ \tilde\varphi_\beta(T) = 0 \mbox{ in } \Omega,\end{array}\right.
		\label{E6.22}\\
		&\int_Q\tilde\mu_\beta(u - \tilde u_\beta)\dx\dt \le 0\quad \forall u \in \tilde{K}_{ad},
		\label{E6.23}\\
		&\gammab\tilde\varphi_\beta + \gammab^2\kappa\tilde u_\beta + \gammab\bar\varphi + \gammab^2\kappa\bar u - \beta_4+ \tilde\mu_\beta = 0.
		\label{E6.24}
	\end{align}
	Arguing similarly as in the proof of Theorem \ref{T4.4} we deduce that $\tilde u_\beta$ and $\tilde\mu_\beta$ belong to the space $L^\rd(0,T;L^2(\Omega))$. Thus, $\tilde u_\beta$ is the unique solution of (P$_\beta$). Moreover, from \eqref{E6.21} and \eqref{E6.22} along with \eqref{E6.8} we infer that $\tilde y_\beta \in \mathcal{Y}$ and $\tilde\varphi_\beta \in \Phi$. Hence, we have that $(\tilde y_\beta,\tilde\varphi_\beta,\tilde u_\beta) \in X$ and \eqref{E6.23} holds for every $u \in \kad$. Due to the convexity of (P$_\beta$), we know that \eqref{E6.21}--\eqref{E6.24} are necessary and sufficient conditions of optimality for (P$_\beta$). This fact and the strict convexity of $\mathcal{J}_\beta$ imply that the system \eqref{E6.21}--\eqref{E6.24} has a unique solution $(\tilde y_\beta,\tilde\varphi_\beta,\tilde u_\beta,\tilde\mu_\beta)$. If we set $y_\beta = \tilde y_\beta + \bar y$, $\varphi_\beta = \tilde\varphi_\beta + \bar\varphi$, $u_\beta = \tilde u_\beta + \bar u$, and $\mu_\beta = \tilde\mu_\beta$, \eqref{E6.21}--\eqref{E6.24} yields that $(y_\beta,\varphi_\beta,u_\beta)$ is the unique element of $X$ satisfying \eqref{E6.17}.
	
	Now, we prove that this solution is Lipschitz with respect to $\beta$. First, we observe that \eqref{E6.24} can be written as
	\begin{equation}
		\gammab\varphi_\beta + \gammab^2\kappa u_\beta - \beta_4 + \mu_\beta = 0.
		\label{E6.25}
	\end{equation}
	Given $\beta, \hat\beta \in Z$, we infer from \eqref{E6.23}-\eqref{E6.24} and \eqref{E6.25} for $\beta$ and $\hat\beta$
	\begin{align*}
		&\int_\Omega(\gammab\varphi_\beta(t) + \gammab^2\kappa u_\beta(t) - \beta_4(t))(u_{\hat\beta}(t) - u_\beta(t))\dx\dt \le 0,\\
		&\int_\Omega(\gammab\varphi_{\hat\beta}(t) + \gammab^2\kappa u_{\hat\beta}(t) - \hat\beta_4(t))(u_\beta(t) - u_{\hat\beta}(t))\dx\dt \le 0.
	\end{align*}
	Adding these inequalities we get
	\begin{align}
		&\gammab^2\kappa\|u_{\hat\beta}(t) - u_\beta(t)\|^2_{L^2(\Omega)} \le \gammab\int_\Omega(\varphi_\beta(t) - \varphi_{\hat\beta}(t))(u_{\hat\beta}(t) - u_\beta(t))\dx\notag\\
		&  + \|\hat\beta_4(t) - \beta_4(t)\|_{L^2(\Omega)}\|u_{\hat\beta}(t) - u_\beta(t)\|_{L^2(\Omega)}\label{E6.26}
	\end{align}
	for almost every $t \in (0,T)$.
	Now, taking into account that $y_{\hat\beta} - y_\beta = \tilde y_{\hat\beta} - \tilde y_\beta$, $\varphi_{\hat\beta} - \varphi_\beta = \tilde\varphi_{\hat\beta} - \tilde\varphi_\beta$, and $u_{\hat\beta} - u_\beta = \tilde u_{\hat\beta} - \tilde u_\beta$, subtracting the equations \eqref{E6.21} satisfied by $y_{\hat\beta}$ and $y_\beta$, and the equations \eqref{E6.22} for $\varphi_{\hat\beta}$ and $\varphi_\beta$, respectively, we obtain
	\begin{align*}
		&\ \ \gammab\int_Q(\varphi_\beta(t) - \varphi_{\hat\beta}(t))(u_{\hat\beta}(t) - u_\beta(t))\dx\\
		& = \int_Q\left\{\Big(\frac{\partial}{\partial t} + A + \frac{\partial a}{\partial y}(x,t,\bar y)\Big)(y_{\hat\beta} - y_\beta)(\varphi_\beta - \varphi_{\hat\beta}) - (\hat\beta_2 - \beta_2)(\varphi_\beta - \varphi_{\hat\beta})\right\} \dx\dt\\
		&= \int_Q\left\{\Big(-\frac{\partial}{\partial t} + A^* + \frac{\partial a}{\partial y}(x,t,\bar y)\Big)(\varphi_\beta - \varphi_{\hat\beta})(y_{\hat\beta} - y_\beta)  - (\hat\beta_2 - \beta_2)(\varphi_\beta - \varphi_{\hat\beta})\right\}\dx\dt\\
		& \hspace{1cm}- \int_\Omega(\hat\beta_1 - \beta_1)(\varphi_\beta(0) - \varphi_{\hat\beta}(0))\,dx\\
		& = -\int_Q\Big\{\big[1 - \frac{\partial^2a}{\partial y^2}(x,t,\bar y)\bar\varphi\big](y_{\hat\beta} - y_\beta)^2 + (\hat\beta_2 - \beta_2)(\varphi_\beta - \varphi_{\hat\beta})\Big\}\dx\dt\\
		& -\int_Q(\hat\beta_3 - \beta_3)(y_{\hat\beta} - y_\beta)\dx\dt - \int_\Omega(\hat\beta_1 - \beta_1)(\varphi_\beta(0) - \varphi_{\hat\beta}(0))\dx.
	\end{align*}
	Let us denote by $\xi_\beta$ and $\xi_{\hat\beta}$ the solutions of \eqref{E6.20} corresponding to $(\beta_1,\beta_2)$ and $(\hat\beta_1,\hat\beta_2)$, respectively. Then, we have that $y_{\hat\beta} - y_\beta = \gammab G'(\bar u)(u_{\hat\beta} - u_\beta) + \xi_{\hat\beta} - \xi_\beta = \gammab z_{u_{\hat\beta} - u_\beta} + (\xi_{\hat\beta} - \xi_\beta)$. Inserting this identity in the above equality we infer
	\begin{align*}
		&\ \ \gammab\int_Q(\varphi_\beta(t) - \varphi_{\hat\beta}(t))(u_{\hat\beta}(t) - u_\beta(t))\dx\\
		& = -\gammab^2\int_Q\big[1 - \frac{\partial^2a}{\partial y^2}(x,t,\bar y)\bar\varphi\big]z^2_{u_{\hat\beta} - u_\beta}\dx\dt\\
		& - \int_Q\big[1 - \frac{\partial^2a}{\partial y^2}(x,t,\bar y)\bar\varphi\big][(\xi_{\hat\beta} - \xi_\beta)^2 + 2z_{u_{\hat\beta} - u_\beta}(\xi_{\hat\beta} - \xi_\beta)]\dx\dt\\
		& - \int_Q\Big\{(\hat\beta_2 - \beta_2)(\varphi_{\hat\beta} - \varphi_\beta) + (\hat\beta_3 - \beta_3)(y_{\hat\beta} - y_\beta)\Big\}\dx\dt\\
		&\hspace{1cm} - \int_\Omega(\hat\beta_1 - \beta_1)(\varphi_\beta(0) - \varphi_{\hat\beta}(0))\dx\\
		& \le -\gammab^2\int_Q\big[1 - \frac{\partial^2a}{\partial y^2}(x,t,\bar y)\bar\varphi\big]z^2_{u_{\hat\beta} - u_\beta}\dx\dt\\
		& + C_3\Big\{\|\hat\beta_1 - \beta_1\|^2_{L^2(\Omega)} + \|\hat\beta_2 - \beta_2\|^2_{L^2(Q)}\\
		&\hspace{1cm} + \|u_{\hat\beta} - u_\beta\|_{L^2(Q)}\big[\|\hat\beta_1 - \beta_1\|_{L^2(\Omega)} + \|\hat\beta_2 - \beta_2\|_{L^2(Q)}\big]\\
		&\hspace{1cm} + \|\hat\beta_2 - \beta_2\|_{L^2(Q)}\|\varphi_{\hat\beta} - \varphi_\beta\|_{L^2(Q)} + \|\hat\beta_3 - \beta_3\|_{L^2(Q)}\|y_{\hat\beta} - y_\beta\|_{L^2(Q)}\\
		&\hspace{1cm} + \|\hat\beta_1 - \beta_1\|_{L^2(\Omega)}\|\varphi_\beta(0) - \varphi_{\hat\beta}(0)\|_{L^2(\Omega)}\Big\}.
	\end{align*}
	Now, from the equations satisfied by $y_{\hat\beta} - y_\beta$ and $\varphi_{\hat\beta} - \varphi_\beta$ we get
	\begin{align}
		&\|y_{\hat\beta} - y_\beta\|_{W(0,T)} \le C_4\Big(\|u_{\hat\beta} - u_\beta\|_{L^2(Q)} + \|\hat\beta_2 - \beta_2\|_{L^2(Q)} + \|\hat\beta_1 - \beta_1\|_{L^2(\Omega)}\Big), \label{E6.27}\\
		&\|\varphi_{\hat\beta} - \varphi_\beta\|_{H^1(Q)} \le C_5\Big(\|y_{\hat\beta} - y_\beta\|_{L^2(Q)} + \|\hat\beta_3 - \beta_3\|_{L^2(Q)}\Big).\label{E6.28}
	\end{align}
	Using the continuous embeddings $W(0,T) \subset L^2(Q)$ and $H^1(Q) \subset C([0,T];L^2(\Omega))$, and the estimates \eqref{E6.27} and \eqref{E6.28}, we infer
	\begin{align*}
		&\ \ \gammab\int_Q(\varphi_\beta(t) - \varphi_{\hat\beta}(t))(u_{\hat\beta}(t) - u_\beta(t))\dx\\
		& \le -\gammab^2\int_Q\big[1 - \frac{\partial^2a}{\partial y^2}(x,t,\bar y)\bar\varphi\big]z^2_{u_{\hat\beta} - u_\beta}\dx\dt\\
		& + C_6\Big\{\|u_{\hat\beta} - u_\beta\|_{L^2(Q)}\big[\|\hat\beta_1 - \beta_1\|_{L^2(\Omega)} + \|\hat\beta_2 - \beta_2\|_{L^2(Q)} + \|\hat\beta_3 - \beta_3\|_{L^2(Q)}\big]\\
		&\hspace{1cm} + \|\hat\beta_1 - \beta_1\|^2_{L^2(\Omega)} + \|\hat\beta_2 - \beta_2\|^2_{L^2(Q)} + \|\hat\beta_3 - \beta_3\|^2_{L^2(Q)}\Big\}.
	\end{align*}
	Combining this inequality with \eqref{E6.26} and using \eqref{E6.9} we deduce
	\begin{align*}
		&\gammab^2\nu\|u_{\hat\beta} - u_\beta\|^2_{L^2(Q)} \le \gammab^2J''(\bar u)(u_{\hat\beta} - u_\beta)^2\\
		&= \gammab^2\Big\{\int_Q\big[1 - \frac{\partial^2a}{\partial y^2}(x,t,\bar y)\bar\varphi\big]z^2_{u_{\hat\beta} - u_\beta}\dx\dt + \kappa\|u_{\hat\beta} - u_\beta\|^2_{L^2(Q)}\Big\}\\
		&\le C_7\Big\{\|u_{\hat\beta} - u_\beta \|_{L^2(Q)}\Big(\|\hat\beta_1 - \beta_1\|_{L^2(\Omega)} +  \sum_{j = 2}^4\|\hat\beta_j - \beta_j\|_{L^2(Q)}\Big)\\
		& + \|\hat\beta_1 - \beta_1\|^2_{L^2(\Omega)} + \sum_{j = 2}^3\|\hat\beta_j - \beta_j\|^2_{L^2(Q)} \Big\}.
	\end{align*}
	This yields
	\begin{equation}
		\|u_{\hat\beta} - u_\beta\|_{L^2(Q)} \le C_8\|\hat\beta - \beta\|_Z.
		\label{E6.29}
	\end{equation}
	Using \eqref{E6.27} and \eqref{E6.29} it follows that
	\begin{equation}
		\|y_{\hat\beta} - y_\beta\|_{W(0,T)} \le C_9\|\hat\beta - \beta\|_Z.
		\label{E6.30}
	\end{equation}
	Now, \eqref{E6.28} and \eqref{E6.30} lead to
	\begin{equation}
		\|\varphi_{\hat\beta} - \varphi_\beta\|_{H^1(Q)} + \|\varphi_{\hat\beta} - \varphi_\beta\|_{C(\bar Q)} \le C_{10}\|\hat\beta - \beta\|_Z.
		\label{E6.31}
	\end{equation}
	Getting back to \eqref{E6.26}, and using \eqref{E6.31}, we get
	\begin{equation}
		\|u_{\hat\beta} - u_\beta\|_{L^\rd(0,T;L^2(\Omega))} \le C_{11}\|\hat\beta - \beta\|_Z.
		\label{E6.32}
	\end{equation}
	Using this in the equation satisfied by $y_{\hat\beta} - y_\beta$ we also obtain
	\begin{equation}
		\|y_{\hat\beta} - y_\beta\|_{C(\bar Q)} \le C_{12}\|\hat\beta - \beta\|_Z.
		\label{E6.33}
	\end{equation}
	Now, \eqref{E6.30}--\eqref{E6.33} imply \eqref{E6.18}. Hence, we apply \cite[Theorem 2.4]{Dontchev1995} to deduce the existence of $\bar\varepsilon \in (0,\varepsilon]$ and $\tilde\rho \in (0,\bar\rho]$ such that for every $\gamma \in (\gammab - \bar\varepsilon,\gammab + \bar\varepsilon)$ the system \eqref{E6.12}--\eqref{E6.15} has a solution $(y,\varphi,u)$ with $u$ in the interior of the ball $B_{\tilde\rho}(\bar u)$ satisfying \eqref{E6.16}. Since $\bar\varepsilon \le \epsilon$ and $\tilde\rho \le \bar\rho$, we know that $J_\gamma$ is strictly convex on $B_{\tilde\rho}(\bar u)$, hence $u$ is the unique local minimizer of \Qbg in this ball. Moreover, $u_\gamma = \gamma u$ belongs to the interior of the ball $B_\rho(u_\gammab)$ and $u_\gamma$ is the unique local minimizer of \Pbg in $B_\rho(u_\gammab)$. Moreover, from \eqref{E6.16} we infer
	\begin{align*}
		\|u_\gamma - u_\gammab\|_{L^\rd(0,T;L^2(\Omega))} &\le \gamma\|u - \bar u\|_{L^\rd(0,T;L^2(\Omega))} + |\gamma - \gammab|\|\bar u\|_{L^\rd(0,T;L^2(\Omega))}\\
		&< (\gammab + \bar\varepsilon)\lambda|\gamma - \gammab| + |\gamma - \gammab|\|\bar u\|_{L^\rd(0,T;L^2(\Omega))} = L|\gamma - \gammab|
	\end{align*}
	for $L = (\gammab + \bar\varepsilon)\lambda + \|\bar u\|_{L^\rd(0,T;L^2(\Omega))}$. This ends the proof.
	\qed\end{proof}


\end{document}